\documentclass[leqno, 11pt]{amsart}
\usepackage{amssymb}
\usepackage{amsmath, amsfonts,enumerate}
\usepackage{graphics}
\usepackage{verbatim}
\usepackage{amsthm}
\usepackage{latexsym,bm}
\usepackage{euscript}
\usepackage{amscd}

\let\cal\mathcal

\makeatletter \@mparswitchfalse \makeatother
\newtheorem{theorem}{Theorem}
\newtheorem{lemma}[theorem]{Lemma}
\newtheorem{sublemma}[theorem]{Sublemma}
\newtheorem{corollary}[theorem]{Corollary}
\newtheorem{proposition}[theorem]{Proposition}
\theoremstyle{remark}
\newtheorem{remark}[theorem]{Remark}
\newtheorem{remarks}[theorem]{Remarks}
\theoremstyle{definition}
\newtheorem{definition}[theorem]{Definition}
\newtheorem{problem}[theorem]{Problem}

\numberwithin{equation}{section}
\numberwithin{theorem}{section}

\def\Z{\mathbb Z}

\def\M{\cal{M}}
\def\H{\cal{H}}
\let\<\langle
\def\ch{\raise 0.5ex \hbox{$\chi$}}
\def\T{\tau}
\def\E{\cal{E}}

\let\\\cr
\let\phi\varphi

\let\epsilon\varepsilon

\def\log{\operatorname{log}}
\renewcommand{\a}{\alpha}
\renewcommand{\b}{\beta}
\newcommand{\g}{\gamma}

\renewcommand{\O}{{\Omega}}

\newcommand{\N}{\cal{N}}
\newcommand{\h}{\mathsf{h}}

\newcommand{\tr}{\mbox{\rm tr}}
\newcommand{\ot}{\otimes}

\begin{document}

\title[Martingale inequalities]{Martingale inequalities in noncommutative symmetric spaces}
\author[N. Randrianantoanina]{Narcisse Randrianantoanina}
\address{Department of Mathematics, Miami University, Oxford,
Ohio 45056, USA}
 \email{randrin@miamioh.edu}

\author[L. Wu]{Lian Wu}
 \address{School of Probability and Statistics, Central South University, Changsha 410075, China and  Department of Mathematics, Miami University, Oxford, Ohio 45056, USA}
 \thanks{Wu was partially supported by NSFC(No.11471337) and China Scholarship Council}
 \email{wul5@miamioh.edu}
 


\begin{abstract} 
We provide    generalizations of Burkholder's inequalities involving  conditioned square functions  of martingales to the general context of martingales in  noncommutative symmetric spaces. More precisely, we prove that  Burkholder's inequalities  are valid for any   martingale in noncommutative  space constructed from a  symmetric space defined on the interval $(0,\infty)$ with Fatou property and whose  Boyd  indices are  strictly between 1 and  2.   This answers positively  a question raised by Jiao and may be viewed as a conditioned version of similar inequalities for square functions of noncommutative martingales.
Using duality,  we also recover the  previously known case where the Boyd indices are finite and  are strictly larger than 2.
 \end{abstract}

\maketitle


\section{Introduction}
In  classical martingale theory, a fundamental result due to Burkholder (\cite{Bu1, BG, Ga})
can be described as follows: given a probability space $(\O, \Sigma, P)$, let $\{\Sigma_n\}_{n\geq 1}$ be an increasing sequence of $\sigma$-fields  of $\Sigma$ such that $\Sigma= \bigvee \Sigma_n$ . If  $2\leq p<\infty$ and $f=(f_n)_{n\geq 1}$ is  a $L_p$-bounded martingale adapted to the filtration $\{\Sigma_n\}_{n\geq 1}$, then (using the convention that $\Sigma_0=\Sigma_1$),
\begin{equation}\label{Burkholder}
\sup_{n\geq 1} [\mathbb{E}|f_n |^p]^{1/p} \simeq_p \Big[\mathbb{E} \big( \sum_{n\geq 1} \mathbb{E}[|df_n|^2 |\Sigma_{n-1}] \big)^{p/2} \Big]^{1/p}  + \Big[ \sum_{n\geq 1} \mathbb{E}|df_n|^p \Big]^{1/p},
\end{equation}
where $\simeq_p$ means equivalence of norms up to constants depending only on $p$.
The random variable $s(f)=\big(\sum_{n\geq 1} \mathbb{E}[|df_n|^2|\Sigma_{n-1}] \big)^{1/2}$ is  called the conditioned square function of  the martingale $f$ and the equivalence \eqref{Burkholder} is generally referred to as Burkholder's inequalities.  The equivalence \eqref{Burkholder} was established by Burkholder as  the    martingale difference sequence generalizations  of  Rosenthal's inequalities \cite{Ros} which state that if $2\leq p<\infty$ and $(g_n)_{n\geq 1}$ is a sequence of independent mean-zero random variables  in $L_p(\Omega, \Sigma, P)$ then 
\begin{equation}\label{Rosenthal}
\Big(\mathbb{E}\big| \sum_{n\geq 1} g_n \big|^p\Big)^{1/p} \simeq_p \Big( \sum_{n\geq 1} \mathbb{E}|g_n|^2\Big)^{1/2}  + \Big( \sum_{n\geq 1} \mathbb{E}|g_n|^p \Big)^{1/p}.
\end{equation}
Probabilistic inequalities  involving  independent random variables and martingales inequalities  play important roles in many  different  areas of mathematics.
Burkholder/Rosenthal inequalities  in particular have many applications in probability theory  and   structures  of symmetric spaces in Banach space theory. On the other hand,  a recent trend in the general study of martingale inequalities is to  find  analogues of classical  inequalities  in the context of noncommutative $L_p$-spaces. We refer to
 \cite{PX,Ju,Junge-Perrin,Ran15} for additional  information on noncommutative martingale inequalities. 
Noncommutative analogues of \eqref{Burkholder} and \eqref{Rosenthal} were extensively studied by Junge and Xu in \cite{JX, JX3}. They obtained that if  $2\leq p <\infty$ and $x=(x_n)_{n\geq 1}$ is a noncommutative martingale that is $L_p$-bounded  then
\begin{equation}\label{nc1}
\big\|x\big\|_p \simeq_p \max\Big\{ \big\|s_c(x)\big\|_p,  \big\|s_r(x)\big\|_p, \big( \sum_{n\geq 1} \big\|dx_n\big\|_p^p \big)^{1/p}\Big\}
\end{equation}
where $s_c(x)$ and $s_r(x)$ denote the column version and the row version  of  conditioned square functions which we refer to the next section for formal definitions. Moreover, they also treated  the corresponding inequalities for  the range $1<p<2$ which are  dual versions of \eqref{nc1} and read as follows: if $x=(x_n)_{n\geq 1}$ is a noncommutative martingale in $L_2(\M)$ then 
\begin{equation}\label{nc2}
\big\|x \big\|_p \simeq_p \inf\Big\{ \big\|s_c(y)\big\|_p +  \big\|s_r(z)\big\|_p + \big( \sum_{n\geq 1} \big\|dw_n\big\|_p^p \big)^{1/p}\Big\}
\end{equation}
where  the infimum is taken over all $x=y+z +w$  with $y$, $z$, and $w$ are martingales. The differences between the two cases $1<p<2$ and $2\leq p<\infty$ are now well-understood in the field. 
In \cite{Jiao1},   inequalities \eqref{nc1} and \eqref{nc2}  were extended to the case of noncommutative Lorentz  spaces $L_{p,q}(\M)$ for $1<p<\infty$ and $1\leq q<\infty$.
Motivated by this extension, it is natural  to ask if some versions of noncommutative Burkholder's inequalities remain valid in the general context of noncommutative  symmetric spaces. This question was explicitly raised in  \cite[Problem~3.5]{Jiao2}.
Martingale inequalities in the general framework of rearrangement invariant spaces  have long been of interests. For the case of classical martingales, 
we refer reader  to the work of Johnson and Schechtman  \cite{JS1, JS2} and the references therein. For the noncommutative settings, we recall that 
  generalizations of Burkholder-Gundy  inequalities in noncommutative symmetric spaces
were recently established  in \cite{Dirk-Pag-Pot-Suk,Jiao2}, extensions of Junge's noncommutative Doob maximal  inequalities in some symmetric spaces were treated in \cite{Dirksen}.
In a closely related topic, Le Merdy and Sukochev  studied  Rademacher averages on noncommutative symmetric spaces (\cite{LeM-Suk}). These Rademacher averages turn out to  provide  one of the key ingredients  in the solution of Burkholder-Gundy  inequalities in noncommutative symmetric spaces in  \cite{Dirk-Pag-Pot-Suk,Jiao2}.    Naturally, the concept of Boyd indices of symmetric  spaces (\cite{LT}) and various interpolation techniques play significant roles in all the results stated above.

The present paper solves the problem discussed above.  Our main result can be summarized as follows: assume that $E$ is a rearrangement invariant function space on $(0,\infty)$ that satisfies some natural conditions and has nontrivial Boyd indices $1<p_E\leq q_E<\infty$ and $\M$ is a semifinite von Neumann algebra equipped with a faithful normal semifinite trace $\T$. We obtain generalizations of \eqref{nc1} and \eqref{nc2} that read:

If $1<p_E \leq q_E<2$, then 
 \begin{equation}\label{nc3}
\big\|x \big\|_{E(\M)} \simeq_E \inf\Big\{ \big\|s_c(y)\big\|_{E(\M)} +  \big\|s_r(z)\big\|_{E(\M)} +   \big\| (dw_n)_{n\geq 1}\big\|_{E(\M \overline{\otimes} \ell_\infty)} \Big\}
\end{equation}
where as in \eqref{nc2}, the  infimum is taken over all  decompositions $x=y+z +w$  with $y$, $z$, and $w$ are martingales in $E(\M,\T)$.

If $2<p_E\leq q_E<\infty$, then 
\begin{equation}\label{nc4}
\big\|x\big\|_{E(\M)}\simeq_E \max\Big\{ \big\|s_c(x)\big\|_{E(\M)},  \big\|s_r(x)\big\|_{E(\M)},  \big\| (dx_n)_{n\geq 1}\big\|_{E(\M \overline{\otimes} \ell_\infty)}\Big\}.
\end{equation}

We note  that  \eqref{nc4} was   recently established  by Dirksen in \cite{Dirksen2}. His approach  follows closely the  original arguments  used in \cite{JX} taking advantage of the fact   mentioned earlier that the noncommutative Burkholder-Gundy inequalities for square functions are valid for noncommutative martingales in some general symmetric spaces. Thus, our main motivation is  primarily to establish  the equivalence \eqref{nc3}.

Our approach is based  on another discovery  made in the next section that, in some sense, one inequality in  the equivalence \eqref{nc2}  can be achieved with a decomposition that works simultaneously for all $1<p<2$. We refer to Theorem~\ref{simultaneous} below for  more information.  This simultaneous decomposition allows us to efficiently apply results from interpolation theory.  Namely,  we use concrete realization of noncommutative symmetric spaces as interpolations of noncommutative $L_p$-spaces by means of $K$-functionals and $J$-functionals. The non-trivial inequality 
in the equivalence \eqref{nc4} will be deduced from \eqref{nc3} using duality. Unlike the $L_p$-cases, this duality  technique does not seem to apply  for the other direction. That is, at the time of this writing, we lack   necessary  ingredients  to deduce \eqref{nc3} from \eqref{nc4}.

The paper is organized as follows: in Section~2, we  provide  some preliminary results concerning noncommutative symmetric spaces, interpolation theory, and  martingale inequalities.  In particular,  we establish  a decomposition result  for noncommutative martingales that sets up the use  of interpolations.
Section~3 is devoted entirely to the statement  and proof of our main result. In the last section, we discuss some related results,  provide examples, and  point to  related open questions  concerning Burkholder's inequalities.

Our notation and terminology are standard  as may be found in the books \cite{BENSHA,LT, TAK}.


\section{Definitions and preliminary results}

\subsection{Noncommutative spaces}

In this subsection, we review some basic facts on rearrangement invariant spaces and their noncommutative counterparts  that are relevant for our presentation.

For   a semifinite von Neumann algebra $\M$ equipped with a faithful normal semifinite trace $\T$, let $\widetilde{\M}$ denote the topological $*$-algebra of all measurable operators with respect to $(\M,\T)$ in the sense of \cite{N}.  For $x \in \widetilde{\M}$, define its  generalized singular number by 
\[
\mu_t(x)=\inf\{\lambda>0; \T( e^{|x|}(\lambda, \infty))\leq t \}, \quad t>0
\]
where $e^{|x|}$ is the spectral measure of  $|x|$.
The function $t \mapsto \mu_t(x)$ from $(0,\infty)$ into $[0,\infty)$ is right-continuous and nonincreasing (\cite{FK}). For the case where $\M$ is the abelian von Neumann algebra $L_\infty(0,\infty)$ with  the trace given by  integration  with respect to the Lebesgue  measure, $\widetilde{\M}$  becomes  the linear space of all measurable functions $L_0(0,\infty)$ and $\mu(f)$ is the decreasing rearrangement of  the function $|f|$ in the sense of \cite{LT}.

 We recall that a Banach  function space $(E, \|\cdot\|_E)$ on $(0,\infty)$ is  called
 \emph{symmetric} if for any $g \in E$ and any measurable function $f$ with $\mu(f) \leq \mu(g)$, we have $f\in E$  and $\|f\|_E \leq \|g\|_E$.  The K\"othe dual of  a symmetric space $E$ is the function space defined by setting:
 \[
 E^\times=\Big\{ f \in L_0(0,\infty): \int_0^\infty |f(t)g(t)|\  dt <\infty, \forall g \in E\Big\}.
 \] 
When equipped with the norm $\|f\|_{E^\times}:=\sup\{\int_0^\infty |f(t)g(t)|\ dt : \|g\|_E \leq 1\}$, $E^\times$ is a symmetric Banach function space.
  
 The symmetric Banach function space $E$ is said to have the \emph{Fatou property}  
 if, whenever $0\leq f_\alpha \uparrow_\alpha \subseteq E$ is an upwards directed net with $\sup_\alpha \|f_\alpha\|_E <\infty$, it follows that $f=\sup_\alpha f_\alpha$ exists in $E$ and $\|f\|_E=\sup_\alpha\|f_\alpha\|_E$.
  It is well-known that $E$ has the Fatou property if and only if the natural  embedding of $E$ into its K\"othe bidual  $E^{\times\times}$ is a surjective isometry. Examples of symmetric spaces with the Fatou property are separable symmetric  spaces and  duals of separable symmetric spaces.
 
 Another concept  that is central to the paper is  the notion of Boyd indices which we now  introduce. Let $E$ be a symmetric Banach space
on $(0,\infty)$. For $s>0$, the dilation operator $D_s: E \to E$
is defined by setting
\begin{equation*}
D_sf(t)=f(t/s), \qquad t>0, \qquad f \in E.
\end{equation*}
The \emph{ lower and upper Boyd indices } of $E$ are defined by
\begin{equation*}
p_E :=\lim_{s\to \infty}\frac{\log s}{\log\Vert D_{1/s}\Vert}\  \text{and}\  
q_E :=\lim_{s\to 0^+}\frac{\log s}{\log\Vert D_{1/s}\Vert},
\end{equation*}
respectively. It is well-known that $1\leq p_E \leq q_E \leq \infty$ and if $E=L_p$  for $1\leq p \leq
\infty$ then $p_E=q_E=p$. 
 We shall
say that $E$ has non-trivial Boyd indices whenever $1<p_E \leq q_E<\infty$.
  We refer to \cite{BENSHA, LT} for any unexplained terminology from  function space theory.
 
For a given symmetric Banach function space  $(E, \|\cdot\|_E)$ on  the
interval $(0, \infty)$, we define the  corresponding  noncommutative space by setting:
\begin{equation*}
E(\M, \T) = \big\{ x \in
\widetilde{\M}\ : \ \mu(x) \in E \big\}. 
\end{equation*}
Equipped with the norm
$\|x\|_{E(\M,\T)} := \| \mu(x)\|_E$, the space  $E(\M,\T)$ is a complex Banach space (\cite{Kalton-Sukochev}) and is referred to as the \emph{noncommutative symmetric space} associated with $(\M,\T)$ corresponding to  the function space $(E, \|\cdot\|_E)$. 
 We remark that 
 if
$1\leq p<\infty$ and $E=L_p(0, \infty)$ then $E(\M, \T)=L_p(\M,\T)$ is  the usual noncommutative $L_p$-space associated with $(\M,\T)$. 

Recall that a  linear operator $T: X \to Y$ is called a \emph{semi-embedding} if  $T$ is one to one and $T(B_X)$ is a closed subset of $Y$ where $B_X=\{x\in X: \|x\| \leq 1\}$.
As in the commutative case,  if $1\leq p <  p_E \leq q_E <q \leq  \infty$ then the space $E(\M,\T)$ is intermediate to the spaces $L_p(\M,\T)$ and $L_q(\M,\T)$ in the sense that
\[
L_p(\M, \T) \cap  L_q(\M,\T) \subseteq E(\M,\T) \subseteq L_p(\M,\T) + L_q(\M, \T)
\]
with the inclusion maps being continuous. Moreover, if $E$ satisfies  the Fatou property, one can readily  verify that the second  inclusion  map  $E(\M,\T) \hookrightarrow L_p(\M,\T) + L_q(\M,\T)$ is a semi-embedding. These facts will be used in the sequel.

We end this subsection with  the following elementary lemma.  It will be used in the proof of our main result. We include a proof for completeness. 
\begin{lemma}\label{approximation2} Assume that $1<p<q<2$ and let $u \in L_p(\M) \cap L_q(\M)$. There exists a sequence $(u_m)_{m\geq 1}$ in $L_1(\M) \cap L_2(\M)$  with  $\lim_{m\to \infty}\|u_{m}-u\|_{L_p(\M) \cap L_q(\M)}=0$  and both sequences $(\|u_m\|_p)_{m\geq 1}$ and 
$(\|u_m\|_q)_{m\geq 1}$ are increasing and converge to $\|u\|_p$ and $\|u\|_q$,  respectively. 
\end{lemma}
\begin{proof} Let $e^{|u|}$ denote the spectral measure of $|u|$. For each $m\geq 1$, set $e_m:=e^{|u|}([1/m, m])$. Then $(e_m)_m$ is an increasing sequence of projections in $\M$ that converges to ${\bf 1}$ for the strong operator topology. 
Clearly, for every $m\geq 1$, $\T(e_m)\leq m^p \|u\|_p^p<\infty$. Set $u_m:= ue_m$.
It is easy to check that $u_m \in L_1(\M) \cap \M$. A fortiori, $u_m \in L_1(\M) \cap L_2(\M)$.  From the identities
\begin{align*}
\max\big\{\|u-u_m\|_p, &\|u-u_m\|_q\big\} =\max\big\{\|u({\bf 1}-e_m)\|_p, \|u({\bf 1}-e_m)\|_q\big\} \\
&=\max\big\{\|({\bf 1}-e_m)|u|^2({\bf 1}-e_m)\|_{p/2}^{1/2},\|({\bf 1}-e_m)|u|^2({\bf 1}-e_m)\|_{q/2}^{1/2}\big\},
\end{align*}
 we get that $\lim_{n\to \infty} \| u-u_m\|_{L_p(\M) \cap L_q(\M)} =0$. 
On the other hand, if  $s$ is equal to either $p$ or $q$, it follows from   the identity
$\|u_m\|_s=\||u| e_m |u| \|_{s/2}^{1/2}$   that $(\|u_m\|_s)_m$ forms an increasing sequence that converges to $\|u\|_s$.
\end{proof}

We refer to \cite{CS, DDP1, PX3, X} for extensive discussions on various properties of 
noncommutative symmetric spaces.


\subsection{Function spaces and interpolations}
In this subsection, we will discuss concrete description of certain classes of noncommutative symmetric spaces as interpolations of noncommutative $L_p$-spaces that are relevant for our method of   proof in the next section. We  begin by recalling that for a given compatible Banach couple  $(X_0, X_1)$, a Banach space $Z$ is called  an \emph{interpolation space} if $X_0 \cap X_1 \subseteq Z \subseteq X_0 +X_1$ and whenever  a bounded linear operator $T: X_0 + X_1 \to X_0 +X_1$ is such that $T(X_0) \subseteq X_0$ and $T(X_1) \subseteq X_1$ we have $T(Z)\subseteq Z$
and $\|T:Z\to Z\|\leq C\max\{\|T :X_0 \to X_0\| , \|T:X_1 \to X_1\|\}$ for some constant $C$. In this case,  we write $Z \in {\rm Int}(X_0, X_1)$. When $C=1$, $Z$ is called exact interpolation space.  We refer to \cite{BENSHA,BL,KaltonSMS} for more on interpolations.

 In this paper we rely heavily on   the  notions of $K$-functionals and
$J$-functionals which we now review:

For a compatible  Banach couple $(X_0, X_1)$, we define  the $J$-functional by setting for any $x \in X_0
\cap X_1$ and $t>0$, 
$$
J(x,t;X_0, X_1)=\max\big\{\|x\|_{X_0}, t\|x\|_{X_1}\big\}.
$$
As a dual notion, the $K$-functional is defined by setting for any $x\in E_0 +E_1$ and $t>0$,
\[
K(x,t;X_0, X_1)=\inf\big\{\|x_1\|_{X_0} + t\|x_2\|_{X_1}; x=x_1 +x_2\big\}.
\]
If the compatible couple $(X_0, X_1)$ is clear from the context, then
we will simply write $J(x,t)$ and $K(x, t)$  in place of  $J(x,t;X_0, X_1)$ and $K(x,t;X_0, X_1)$,  respectively. 
It is now quite well-known that  any   symmetric Banach function space  with the Fatou property that belongs to  ${\rm Int}(L_p, L_q)$ is given by a  $K$-method.  More precisely,  we have the following result due to  Brudnyi and Krugliak (see for instance \cite[Theorem~6.3]{KaltonSMS}).

\begin{theorem} Let $E$ be a  symmetric Banach function space on $(0,\infty)$ with the Fatou property. If  $E \in {\rm Int}(L_p(0,\infty), L_q(0,\infty))$ for $1\leq p<q\leq \infty$, then there exists a function space $F$ on $(0,\infty)$ such that $f \in E$ if and only if $K(f, \cdot, L_p,L_q) \in F$ and there exists a constant   $C$ such that
\[
C^{-1} \big\| K(f,\cdot)\big\|_F \leq \|f\|_E \leq C \big\| K(f,\cdot)\big\|_F.
\]
\end{theorem}
We will use the corresponding $J$-method of the above theorem. This was  studied in \cite{Bennett1,Bennett2}. We  review the basic construction of this method and introduce a discrete version that is quite essential in the next section.

Suppose that an element $x \in X_0 +X_1$  admits a representation 
\begin{equation}\label{rep}
x=\int_0^\infty  u(t)\ dt/t, 
\end{equation}
where $u(\cdot)$ is measurable function that takes its  values in $X_0 \cap X_1$ and the  integral is convergent in $X_0 + X_1$. For any given representation $u(\cdot)$, we set for $s>0$,
\begin{equation}
j(u,s)=\int_s^\infty t^{-1} J(u(t), t)\ dt/t.
\end{equation}
Given a symmetric Banach function space $F$ defined on $(0,\infty)$, the interpolation space $(X_0, X_1)_{F, j}$  consists of elements  $x \in X_0 +X_1$ which admit a representation  as in \eqref{rep} 
 and are such that
\begin{equation}
\big\| x \big\|_{F,j} =\inf\Big\{ \big\| j(u,\cdot)\big\|_F\Big\} <\infty,
\end{equation}
where the infimum is taken over all representation $u$ of $x$ as in \eqref{rep}. 
We refer to \cite{Bennett1, Bennett2} for  a comprehensive study of this interpolation method along with some other equivalent methods.
As noted above,  we will need a discrete version  of this  method.  This is standard but we could not find any reference in the literature for this particular method so we provide the details.

 We define  the interpolation space $(X_0, X_1)_{F, \underline{j}}$ to be the space of elements $x\in X_0 +X_1$ which admit a representation
\begin{equation}\label{series}
x=\sum_{\nu \in \mathbb{Z} }u_\nu \quad \text{(convergence in $X_0 +X_1$)}
\end{equation}
with $u_\nu \in X_0 \cap X_1$ and are such that
\[
\big\|x\big\|_{F, \underline{j}}=\inf\Big\{ \Big\| \underline{j}(\{u_\nu\}_\nu, \cdot)\Big\|_F\Big\} <\infty,
\]
where  the decreasing  function $\underline{j}(\{u_\nu\}_\nu, \cdot)$ is defined by
\[
\underline{j}(\{u_\nu\}_\nu, t) =\sum_{\gamma\geq \nu +1} 2^{-\gamma} J(u_\gamma, 2^{\gamma})  \quad \text{for}\ t\in [2^\nu, 2^{\nu+1}),
\]
and
the infimum is taken over all representations of $x$ as in \eqref{series}.  Clearly, the function $\underline{j}(\{u_\nu\}_\nu, t)$ takes only countably many values. Thus, we may call  $(X_0, X_1)_{F,\underline{j}}$   as a discrete interpolation method.
As in the case of real interpolation methods, this discrete version is equivalent to the continuous version described earlier.  More precisely,  we have:
\begin{lemma}\label{discrete} Let $x \in X_0 +X_1$. Then 
 $x \in (X_0, X_1)_{F, j}$ if and only if $x \in (X_0, X_1)_{F, \underline{j}}$. More precisely,  the following inequalities hold: 
\[
\frac{1}{4}\big\|x\big\|_{F,j} \leq \big\|x\big\|_{F,\underline{j}} \leq 4 \big\|x\big\|_{F,j}.
\]
\end{lemma}
The verification of Lemma~\ref{discrete} is a simple adaptation of standard arguments from interpolation theory which we leave for the reader (see  \cite{BL}).

Combining  \cite[Theorem~9.3]{Bennett1},  \cite[Theorem~3.5]{Bennett2},  and Lemma~\ref{discrete}, we may state the following result which is one of the decisive tools we use in our proof.
\begin{theorem} If $E$ is a  symmetric Banach  function space on $(0,\infty)$  with the Fatou property then the following are equivalent:
\begin{itemize}
\item[(i)]  $1<p<p_E \leq q_E <q<\infty$.
\item [(ii)] There exists a symmetric Banach  function space $F$ on $(0,\infty)$  with nontrivial Boyd indices  such that:
\[
E=(L_p(0,\infty), L_q(0,\infty))_{F,\underline{j}}\quad \text{(with equivalent norms)}.
\]
\end{itemize}
\end{theorem}

As is now well-known, the preceding  interpolation result automatically lifts to the noncommutative setting (see  \cite[Corollary~2.2]{PX3}):
\begin{corollary}\label{lifted-interpolation}
Let   $E$  be a  symmetric Banach  function space on $(0,\infty)$ with the Fatou property.  Then the following are equivalent:
\begin{itemize}
\item[(i)]  $1<p<p_E \leq q_E <q<\infty$.
\item [(ii)] There exists a symmetric Banach  function space $F$ on $(0,\infty)$  with nontrivial Boyd indices  such that for every semifinite  von Neumann algebra $(\N,\sigma)$,
\[
E(\N,\sigma)=(L_p(\N,\sigma), L_q(\N,\sigma))_{F,\underline{j}},
\]
with  equivalent norms depending only on $E$, $p$, and $q$.
\end{itemize}
\end{corollary}

We record   a general fact about interpolations of linear operators between two noncommutative spaces for further use.
\begin{proposition}[\cite{DDP4}]\label{Operator-interpolation} Assume that $E \in {\rm Int}(L_p, L_q)$ and  $\M$ and $\N$ are semifinite von Neumann algebras. Let $T: L_p(\M) + L_q(\M) \to L_p(\N) +L_q(\N)$ be a linear operator such that $T: L_p(\M) \to L_p(\N)$ and $T: L_q(\M) \to L_q(\N)$ are bounded. Then $T$ maps $E(\M)$ into $E(\N)$ and the resulting operator $T: E(\M) \to E(\N)$ is bounded.  Moreover, we have the following estimate:
\[
\big\| T:E(\M) \to E(\N)\big\| \leq C \max\big\{ \big\| T: L_p(\M) \to L_p(\N)\big\|, \big\| T: L_q(\M) \to L_q(\N)\big\| \big\}
\]
for some absolute constant $C$.
\end{proposition}

\subsection{Noncommutative martingales}

Let us now recall the general setup for noncommutative martingales.
In the sequel, we always assume that the von Neumann algebra $\M$ is such that $\M_*$ is separable.  Denote by $(\M_n)_{n \geq 1}$ an
increasing sequence of von Neumann subalgebras of ${\M}$
whose union  is weak*-dense in
$\M$. For $n\geq 1$, we assume that there exists a trace preserving conditional expectation ${\E}_n$ 
from ${\M}$ onto  ${\M}_n$.  It is well-known that if  $\T_n$  denotes the restriction of $\T$ on $\M_n$, then $\E_n$ extends to a contractive projection from $L_p(\M,\T)$ onto $L_p(\M_n, \T_n)$ for all $1\leq p \leq \infty$. More generally, if $E$ is a symmetric Banach function space on $(0,\infty)$  which is an  interpolation space of  the couple $(L_1(0,\infty), L_\infty(0,\infty))$ then $\E_n$  is bounded  from $E(\M,\T)$ onto $E(\M_n,\T_n)$.  

\begin{definition}
A sequence $x = (x_n)_{n\geq 1}$ in $L_1(\M)$ is called \emph{a
noncommutative martingale} with respect to $({\M}_n)_{n \geq
1}$ if $\mathcal{E}_n (x_{n+1}) = x_n$ for every $n \geq 1.$
\end{definition}
If in addition, all $x_n$'s belong to $E(\M)$  then  $x$ is called an $E(\M)$-martingale.
In this case we set
\begin{equation*}\| x \|_{E(\M)} = \sup_{n \geq 1} \|
x_n \|_{E(\M)}.
\end{equation*}
If $\| x \|_{E(\M)} < \infty$, then $x$ is called
a bounded $E(\M)$-martingale. 

Let $x = (x_n)_{n\ge1}$ be a noncommutative martingale with respect to
$(\M_n)_{n \geq 1}$.  Define $dx_n = x_n - x_{n-1}$ for $n
\geq 1$ with the usual convention that $x_0 =0$. The sequence $dx =
(dx_n)_{n\ge1}$ is called the \emph{ martingale difference sequence} of $x$. A martingale
$x$ is called  \emph{a finite martingale} if there exists $N$ such
that $d x_n = 0$ for all $n \geq N.$
In the sequel, for any operator $x\in E(\M)$,  we denote $x_n=\E_n(x)$ for $n\geq 1$.
It is well-known that if $1<p<\infty$ and $x=(x_n)_{n\geq 1}$ is a bounded $L_p(\M)$-martingale then there exists $x_\infty \in L_p(\M)$ such that $x_n=\E_n(x_\infty)$ (for all $n\geq 1$) and $\|x\|_{p}=\|x_\infty\|_p$.  For the   general context of symmetric spaces,  if $E\in {\rm Int}(L_p,L_q)$ for $1<p\leq q <\infty$ and  satisfies the Fatou property, then any  bounded $E(\M)$-martingale $x=(x_n)_{n\geq 1}$ is of the form $(\E_n(x_\infty))_{n\geq 1}$ where  $x_\infty \in E(\M)$ satisfying $\|x\|_{E(\M)} \simeq_E \|x_\infty\|_{E(\M)}$, with equality if $E$ is an exact interpolation space. Indeed,  by reflexivity, such statement can be readily verified for $L_p(\M) + L_q(\M)$ and  for general  such $E$, it  follows from the fact that $E(\M)$ semi-embeds into $L_p(\M) +L_q(\M)$. Because of these facts, we often identify martingales with  measurable operators when appropriate.

Let us now  review the definitions of  Hardy spaces  and conditioned Hardy spaces of noncommutative martingales.  Throughout, 
 $E$ is a  symmetric function space on $(0,\infty)$ with the additional properties that it satisfies the Fatou property and  $E\in {\rm Int}(L_1, L_\infty)$. 
 
  We define 
the column space $E(\M;\ell_2^c)$ to be the linear space  of all sequences $(a_n)_{n\geq 1} \subset E(\M)$ such that the  infinite column matrix $\sum_{n\geq 1} a_n \otimes e_{n,1} \in E(\M  \overline{\otimes} B(\ell_2(\mathbb{N})))$ where $B(\ell_2(\mathbb{N}))$ is the algebra of all bounded operators  on the Hilbert space $\ell_2(\mathbb{N})$ equipped with its usual trace $\tr$ and $(e_{n,m})_{n,m\geq 1}$ denotes the collection of all unit matrices in $B(\ell_2(\mathbb{N}))$.  It is not difficult to verify that the linear subspace $\{\sum_{n\geq 1} a_n \otimes e_{n,1}: (a_n)_n \in E(\M;\ell_2^c)\}$ is closed in $E(\M \overline{\otimes} B(\ell_2))$and therefore
we equip $E(\M;\ell_2^c)$ with the norm
\[
\Big\| (a_n)_{n\geq 1}\Big\|_{E(\M; \ell_2^c)} :=\Big\| \sum_{n\geq 1} a_n \otimes e_{n,1}\Big\|_{E(\M \overline{\otimes} B(\ell_2))}=\Big\| \big(\sum_{n\geq 1} |a_n|^2 \big)^{1/2} \Big\|_{E(\M)},
\]
then it  becomes  a Banach space. We can also define the corresponding row space $E(\M;\ell_2^r)$ to be the space of all sequences $(a_n)_{n\geq 1} \subset E(\M)$ for which $(a_n^*)_{n\geq 1} \in E(\M; \ell_2^c)$ equipped with the norm 
\[
\Big\| (a_n)_{n\geq 1}\Big\|_{E(\M; \ell_2^r)} =\Big\| (a_n^*)_{n\geq 1}\Big\|_{E(\M; \ell_2^c)}.
\]

Following \cite{PX}, we consider  the column and row versions of square functions
relative to a  martingale $x = (x_n)_{n\ge1}$ as follows:
 \[
 S_c (x) = \Big ( \sum_{k \geq1} |dx_k |^2 \Big )^{1/2}
 \
\text{and}
 \ \
 S_r (x) = \Big ( \sum_{k \geq 1} | dx^*_k |^2 \Big)^{1/2},
 \]
where convergences can be taken with respect to the measure topology.
Define $\mathcal{H}_E^c (\mathcal{M})$
(respectively,  $\mathcal{H}_E^r (\mathcal{M})$) to  be the set of all martingales $x=(x_n)_{n\geq 1}$ in $E(\M)$ for which the martingale difference sequence $(dx_n)_{n\geq 1} \in E(\M; \ell_2^c)$ (respectively,  $E(\M;\ell_2^r)$)
 equipped with  the norm $\| x \|_{\mathcal{H}_E^c}=\| (dx_n)_{n\geq 1}\|_{E(\M;\ell_2^c)}$
(respectively,  $\| x \|_{\mathcal{H}_E^r}=\| (dx_n)_{n\geq 1} \|_{E(\M; \ell_2^r)} $).
It is easy to verify that  since conditional expectations are bounded in $E(\M)$,  the normed spaces $\mathcal{H}_E^c (\mathcal{M})$  and $\mathcal{H}_E^r (\mathcal{M})$ embed  isometrically into $E(\M \overline{\otimes} B(\ell_2(\mathbb{N})))$ with closed ranges and therefore they are  Banach spaces. Moreover, one can see  from  its  definition that $\mathcal{H}_E^c (\mathcal{M})$ is simply the space of all martingales $x=(x_n)_{n\geq 1}$ in $E(\M)$  for which $\| x \|_{\mathcal{H}_E^c}= \| S_c(x)\|_{E(\M)} <\infty$. Similar statement is also valid for $\mathcal{H}_E^r(\M)$. 

We now turn to the mixture Hardy spaces of noncommutative martingales.  From the above discussions,  the spaces  $\mathcal{H}_E^c (\mathcal{M})$ and $\mathcal{H}_E^r (\mathcal{M})$  are compatible in the sense that they  continuously embed into the larger space $E(\M \overline{\otimes} B(\ell_2(\mathbb{N})))$. The Hardy space $\mathcal{H}_E(\M)$ is defined as
follows. For $1\leq p_E \leq q_E<2$, 
\begin{equation*}
\mathcal{H}_E(\mathcal{M})
= \mathcal{H}_E^c (\mathcal{M}) + \mathcal{H}_E^r(\mathcal{M})
\end{equation*}
equipped with the  norm
\begin{equation*}
\| x \|_{\mathcal{H}_E} =
\inf \big \{ \| y\|_{\mathcal{H}_E^c} + \| z \|_{\mathcal{H}_E^r} \big\},
\end{equation*}
where the infimum is taken over all
$y \in\mathcal{H}_E^c (\mathcal{M} )$ and $z \in \mathcal{H}_E^r(\mathcal{M} )$
such that $x = y + z$.
For $2 \leq p_E\leq q_E  <\infty$,
\begin{equation*}
\mathcal{H}_E (\mathcal{M}) =
\mathcal{H}_E^c (\mathcal{M}) \cap \mathcal{H}_E^r(\mathcal{M})
\end{equation*}
equipped with the norm
\begin{equation*}
\| x \|_{\mathcal{H}_E} =
\max \big \{ \| x\|_{\mathcal{H}_E^c} ,\; \| x \|_{\mathcal{H}_E^r} \big\}.
\end{equation*}
These definitions mirror the well-documented difference between the two cases  $1\leq p < 2$ and $2 \leq p  < \infty$  for  the special case where $E=L_p(0,\infty)$. We refer to \cite{Dirk-Pag-Pot-Suk} and \cite{Jiao2} for more information and results related to space $\mathcal{H}_E(\M)$.

We now consider  the conditioned versions of the above definitions. Our approach is based on  the conditioned spaces introduced by Junge in \cite{Ju}.  Since this is very crucial in the sequel, we review the basic setup. Below, we  use the convention that $\E_0=\E_1$.

Let $\E:\M \to \N$ be a normal faithful conditional expectation, where $\N$ is a von Neumann subalgebra of $\M$.  For $0<p\leq \infty$, we define  the conditioned space    $L_p^c(\M,\E)$ to be the completion of  $\M \cap L_p(\M)$ with respect to the quasi-norm
\[
\big\|x\big\|_{L_p^c(\M,\E)} =\big\|\E(x^*x)\big\|_{p/2}^{1/2}.
\]
It was shown in \cite{Ju} that for every $n$ and $0<p \leq \infty$,  there exists an isometric right $\M_n$-module map $u_{n,p}: L_p^c(\M, \E_n) \to L_p(\M_n;\ell_2^c)$ such that
\begin{equation}\label{u}
u_{n,p}(x)^* u_{n,q}(y)=\E_n(x^*y) \otimes e_{1,1},
\end{equation}
for all $x\in L_p^c(\M;\E_n)$ and $y \in L_q^c(\M;\E_n)$ with $1/p +1/q \leq 1$.
We now consider  the increasing sequence of expectations $(\E_n)_{n\geq 1}$.  Denote by $\mathcal{F}$ the collection of all finite sequences $(a_n)_{n\geq 1}$ in $L_1(\M) \cap \M$.
 For $0< p\leq \infty$,  define  the space $L_p^{\rm cond}(\M; \ell_2^c)$ to be   the completion  of $\mathcal{F}$  with respect to the norm:
\begin{equation}\label{conditioned-norm}
\big\| (a_n) \big\|_{L_p^{\rm cond}(\M; \ell_2^c)} = \big\| \big(\sum_{n\geq 1} \E_{n-1}|a_n|^2\big)^{1/2}\big\|_p.
\end{equation}
The space $L_p^{\rm cond}(\M;\ell_2^c)$ can be isometrically embedded into an $L_p$-space associated to  a semifinite von  Neumann algebra  by means of the following map: 
\[
U_p : L_p^{\rm cond}(\M; \ell_2^c) \to L_p(\M \overline{\otimes} B(\ell_2(\mathbb{N}^2)))\]
defined by setting 
\[
U_p((a_n)_{n\geq 1}) = \sum_{n\geq 1} u_{n-1,p}(a_n) \otimes e_{n,1}
\]
where as above, $(e_{i,j})_{i,j \geq 1}$ is the family of unit matrices in $B(\ell_2(\mathbb{N}))$. From \eqref{u}, it follows  
   that if $(a_n)_{n\geq 1} \in L_p^{\rm cond}(\M;  \ell_2^c)$ and $(b_n)_{n\geq 1} \in L_q^{\rm cond}(\M;  \ell_2^c)$  for $1/p +1/q \leq 1$ then
\begin{equation}\label{key-identity}
U_p( (a_n))^* U_q((b_n)) =\big(\sum_{n\geq 1} \E_{n-1}(a_n^* b_n)\big) \otimes e_{1,1} \otimes e_{1,1}.
\end{equation}
In particular,  $\| (a_n) \|_{L_p^{\rm cond}(\M; \ell_2^c)}=\|U_p( (a_n))\|_p$ and 
hence  $U_p$ is indeed an isometry.  We note that $U_p$ is independent of $p$ in the sense of interpolation. Below, we will simply write $U$ for $U_p$.
We refer the reader to \cite{Ju} and \cite{Junge-Perrin} for more details on the preceding construction.

 Now, we  generalize  the notion of conditioned spaces to the setting of symmetric spaces. We consider  the  algebraic linear  map $U$ restricted to   the linear space $\mathcal{F}$  that takes  its values in $L_1(\M \overline{\otimes} B(\ell_2(\mathbb{N}^2))) \cap \M \overline{\otimes} B(\ell_2(\mathbb{N}^2))$.  For a given 
   sequence $(a_n)_{n\geq 1} \in \mathcal{F}$, we set:
\[
\big\| (a_n) \big\|_{E^{\rm cond}(\M; \ell_2^c)} = \big\| \big(\sum_{n\geq 1} \E_{n-1}|a_n|^2\big)^{1/2}\big\|_{E(\M)}=\big\| U( (a_n))\big\|_{E(\M \overline{\otimes} B(\ell_2(\mathbb{N}^2)))}.
\]
This is well-defined and  induces   a norm  on the linear space $\mathcal{F}$.  We define   the Banach space $E^{\rm cond}(\M;  \ell_2^c)$ to be the completion of $\mathcal{F}$  with respect to the above  norm. Then  
 $U$ extends to an isometry from  $E^{\rm cond}(\M;\ell_2^c)$ into $E(\M \overline{\otimes} B(\ell_2(\mathbb{N}^2)))$ which we will still denote by $U$.

Similarly, we may define the corresponding  row version $E^{\rm cond}(\M; \ell_2^r)$ which can also be viewed as a subspace of $E(\M \overline{\otimes} B(\ell_2(\mathbb{N}^2)))$ as row vectors. This is done by simply considering adjoint operators.

Now, let $x = (x_n)_{n \geq 1}$ be a finite martingale in $L_2(\M) +\M$.
We set 
 \[
 s_c (x) = \Big ( \sum_{k \geq 1} \E_{k-1}|dx_k |^2 \Big )^{1/2}
 \
\text{and}
 \ \ 
 s_r (x) = \Big ( \sum_{k \geq1} \E_{k-1}| dx^*_k |^2 \Big)^{1/2}.
 \]
These are called the column and row conditioned square functions, respectively.
Let $\mathcal{F}_M$ denote  the set of all finite martingales in $L_1(\M) \cap \M$.
Define  $\h_E^c (\mathcal{M})$ (respectively,  $\h_E^r (\mathcal{M})$) as the completion  of $\mathcal{F}_M$  under the  norm $\| x \|_{\h_E^c}=\| s_c (x) \|_{E(\M)}$
(respectively,  $\| x \|_{\h_E^r}=\| s_r (x) \|_{E(\M)} $). We observe that  for every $x \in \mathcal{F}_M$, $\|x\|_{\h_E^c} =
\| (dx_n)\|_{E^{\rm cond}(\M; \ell_2^c)}$. Therefore, $\h_E^c(\M)$ may be viewed as a subspace of $E^{\rm cond}(\M; \ell_2^c)$.  More precisely, we consider the map
$\mathcal{D}: \mathcal{F}_M \to \mathcal{F}$ by setting $\mathcal{D}(x)= (dx_n)_{n\geq 1}$. Then $\mathcal{D}$ extends  to an isometry from $\h_E^c(\M)$ into $E^{\rm cond}(\M;\ell_2^c)$ which we will  denote by $\mathcal{D}_c$. In the sequel, we will make frequent use of the isometric embedding:
\[
U\mathcal{D}_c: \h_E^c(\M)  \to  
E(\M \overline{\otimes} B(\ell_2(\mathbb{N}^2))).
\]
We can make  similar  assertions for the row case. That is, $\h_E^r(\M)$ embeds isometrically into $E(\M \overline{\otimes} B(\ell_2(\mathbb{N}^2)))$.
We also need the diagonal Hardy space  $\h_E^d(\M)$  which is the space of all martingales whose martingale difference sequences belong to $E(\M \overline{\otimes} \ell_\infty)$ equipped with the norm $\|x\|_{\h_E^d} := \|(dx_n)\|_{E(\M \overline{\otimes} \ell_\infty)}$.  As above, we denote by $\mathcal{D}_d$ the isometric extension of $\mathcal{D}$ from $\h_E^d(\M)$ into $E(\M \overline{\otimes} \ell_\infty)$.  We remark that since, under our assumptions on $E$, conditional expectations are bounded on $E(\M)$, it follows that $\mathcal{D}_d(\h_E^d(\M))$ is a closed subspace of $E(\M \overline{\otimes} \ell_\infty)$. This shows in particular that $\h_E^d(\M)$ is a Banach space.
Using the natural von Neumann algebra embeddings, $\ell_\infty \subset B(\ell_2(\mathbb{N})) \subset B(\ell_2(\mathbb{N}^2))$, we can further state that $\h_E^d(\M)$ embeds isometrically into $E(\M \overline{\otimes} B(\ell_2(\mathbb{N}^2)))$. Consequently, $\h_E^d(\M)$, $\h_E^c(\M)$, and $\h_E^r(\M)$ are compatible as  all three isometrically embed into the  larger Banach space $E(\M \overline{\otimes} B(\ell_2(\mathbb{N}^2)))$.
We define the conditioned version of martingale Hardy spaces as follows.
If $1\leq p_E \leq q_E<2$, then
\begin{equation*}
\h_E(\mathcal{M}) = \h_E^d (\mathcal{M}) +  \h_E^c (\mathcal{M})
+  \h_E^r (\mathcal{M})
\end{equation*}
equipped with the  norm
\begin{equation*}
\| x \|_{\h_E} = \inf \big \{ \| w \|_{ \h_E^d} + \| y
\|_{ \h_E^c} + \| z \|_{ \h_E^r} \big \},
\end{equation*}
where the infimum is taken over all $w \in  \h_E^d
(\M), y \in  \h_E^c (\M)$, and $ z \in \h_E^r
(\M)$ such that $ x = w + y + z.$
 If $2 \leq p_E \leq q_E <\infty$, then
\begin{equation*}\h_E (\M) =  \h_E^d
(\M) \cap  \h_E^c (\M) \cap  \h_E^r (\M)
\end{equation*}equipped with the norm\begin{equation*}
\| x \|_{\h_E} = \max \big \{ \| x \|_{ \h_E^d}, \|x
\|_{ \h_E^c}, \| x \|_{ \h_E^r} \big \}.
\end{equation*}

For the case where $E=L_p(0,\infty)$, we will simply write $\H_p(\M)$, $\h_p(\M)$, ect. in place of $\H_{L_p}(\M)$, $\h_{L_p}(\M)$, ect.   From the noncommutative Burkholder-Gundy inequalities and  noncommutative Burkholder inequalities  proved in \cite{JX, PX}, we have
 \[
 \mathcal{H}_p (\M)= \h_p ({\M}) = L_p({\M})
 \]
with equivalent norms for all $1 < p < \infty$. The latter equality constitutes the primary topic of this paper. 

We collect some basic properties of these various Hardy spaces for further use.
\begin{proposition}\label{comp-Int} Let $1\leq p<q<\infty$ and assume  that $E\in {\rm Int}(L_p, L_q)$. 
\begin{enumerate}[{\rm(i)}]
\item $\h_E^d(\M)$ is complemented in $E(\M \overline{\otimes}\ell_\infty)$ and $\h_E^d(\M) \in {\rm Int}(\h_p^d(\M), \h_q^d(\M))$;
\item If $1<p<q<\infty$, then $\h_E^c(\M)$ is complemented in $E(\M \overline{\otimes} B(\ell_2(\mathbb{N}^2)))$ and 
 for $s\in\{c,r\}$, then $\h_E^s(\M) \in {\rm Int}(\h_p^s(\M), \h_q^s(\M))$.
\end{enumerate}
\end{proposition}

\begin{proof} For the  first item, the complementation  follows immediately from the  simple fact that the map
$\Theta: L_r(\M \overline{\otimes} \ell_\infty)  \to  L_r(\M \overline{\otimes} \ell_\infty)$  defined by:
\[\Theta\big((a_n)_{n\geq 1}\big) = ( \E_{n}(a_n) -\E_{n-1}(a_n))_{n\geq 1}
\]
is  a bounded  projection for all $1\leq r <\infty$. By  the interpolation result stated in Proposition~\ref{Operator-interpolation}, $\Theta: E(\M \overline{\otimes} \ell_\infty)  \to  E(\M \overline{\otimes} \ell_\infty)$ is a bounded projection and it is clear that its range is $\mathcal{D}_d(\h_E^d(\M))$. The interpolation is an obvious consequence of the complementation result.

The second item is also a consequence of the known fact from \cite{Ju} that  if $1<r<\infty$, $\h_r^c(\M)$ is complemented in $L_r(\M \overline{\otimes} B(\ell_2(\mathbb{N}^2)))$.
Indeed, let $\Lambda: L_r(\M \overline{\otimes} B(\ell_2(\mathbb{N}^2))) \to L_r(\M \overline{\otimes} B(\ell_2(\mathbb{N}^2)))$ be the bounded projection whose range is $U\mathcal{D}_c(\h_r^c(\M))$ for all $1 < r<\infty$. It is known that $\Lambda$ is independent of $r$ in the sense of interpolation (\cite{Ju,JX}). We deduce   that  it is also a bounded projection from $E(\M \overline{\otimes} B(\ell_2(\mathbb{N}^2)))$ onto $U\mathcal{D}_c(\h_E^c(\M))$.  The statement about the range comes from the facts that  on one hand, $U\mathcal{D}_c(\mathcal{F}_M) \subset \Lambda\big[E(\M \overline{\otimes} B(\ell_2(\mathbb{N}^2)))\big]$, and on the other hand, 
$\bigcup_{n\geq 1}\big[ L_1(\M_n \overline{\otimes} B(\ell_2(\mathbb{N}^2)))
\cap \M_n \overline{\otimes} B(\ell_2(\mathbb{N}^2))]$ is a  dense subset of $E(\M \overline{\otimes} B(\ell_2(\mathbb{N}^2)))$ whose image  
 $\bigcup_{n\geq 1}\Lambda\big[ L_1(\M_n \overline{\otimes} B(\ell_2(\mathbb{N}^2)))
\cap \M_n \overline{\otimes} B(\ell_2(\mathbb{N}^2))] \subseteq U\mathcal{D}_c(\mathcal{F}_M)$.  This shows that $\Lambda\big[E(\M \overline{\otimes} B(\ell_2(\mathbb{N}^2)))\big]=U\mathcal{D}_c(\h_E^c(\M))$.
As in the first item, the statement on interpolation follows immediately from  the  complementation  result.
\end{proof}
\begin{remark}
Unlike the $L_p$-case,  general  descriptions of the duals of these  more general conditioned Hardy spaces  appear to be unavailable . One of the difficulties that arises in  trying to develop  such  duality theory lies on the
 fact that, in some cases, $E^*$ may not  be a function space. For instance, if $E=L_{r, \infty}$ for $1\leq p <r<q<\infty$, then $E\in {\rm Int}(L_p,L_q)$, has the Fatou property, but $E^*$ is highly nontrivial.
\end{remark}

\medskip

The next theorem is  the main result of this section. Its main feature is  that it gives  a decomposition that provides norms estimates simultaneously  for all  $p\in (1,2)$.
This fact is  very crucial in our approach in the next section. Before formally stating this result, we should clarify that when $1<p<2$, the noncommutative Burkholder inequalities  (\cite[Theorem~6.1]{JX})  imply that for each $w\in \{d,c,r\}$ there exists
a bounded linear map $\xi_p^w: \h_p^w(\M) \to L_p(\M)$. These maps are one to one as they come from the isomorphism $\h_p(\M) \approx L_p(\M)$ and the definition of $\h_p(\M)=\h_p^d(\M) +\h_c^d(\M) +\h_p^r(\M)$. As a result, any element of  $\h_p^w(\M)$ can be uniquely represented as measurable operator from $L_p(\M)$. This justifies the  use of identification in item~{\rm (i)} in the statement  of the  next theorem.

\begin{theorem}\label{simultaneous} There exists a family  $\{\kappa_p : 1<p<2\} \subset \mathbb{R}_+$  satisfying the following:
 if 
 $x\in L_1(\M) \cap L_2(\M)$, then there exist   $a \in \bigcap_{1<p<2} \h_p^d(\M)$, $b \in \bigcap_{1<p<2} \h_p^c(\M)$, and $c \in \bigcap_{1<p<2} \h_p^r(\M)$  such that:
\begin{enumerate}[{\rm(i)}]
\item  $x=a +b+ c$;
\item for every $1<p< 2$, the following inequality holds:
\[
\big\| a \big\|_{\h^d_p}
 + \big\| b \big\|_{\h_p^c} + \big\| c
\big\|_{\h_p^r}\leq \kappa_p \big\| x \big\|_p.
\]
\end{enumerate}
\end{theorem}
\begin{proof}
{\bf Case 1.}  Assume that $\M$ is finite  and $\T$ is a normalized trace.
The proof uses a weak-type decomposition from \cite{Ran21}.  
We consider the interpolation couple $(L_1(\M), L_2(\M))$. 

Let $x \in L_2(\M)$. According to \cite[Lemma~3.3.2]{BL}, there is a representation 
$x =\sum_{\nu \in \mathbb{Z}} u_\nu$ (convergent in $L_1(\M)$)  of $x$ satisfying,    for every $\nu\in \Z$,
\begin{equation}\label{fundamental}
J(u_\nu, 2^\nu) \leq 4 K(x, 2^\nu) .
\end{equation}
 Since $\T({\bf 1})=1$, we may apply   \cite[Theorem~3.1]{Ran21}. The statement of \cite[Theorem~3.1]{Ran21} is only for finite martingales but the construction used there can be applied verbatim to the  case of infinite $L_2(\M)$-bounded martingales.  
 There exists
  an absolute constant $\kappa>0$ such that, for each $\nu \in \mathbb{Z}$,  we can find   three 
adapted sequences $\a^{(\nu)}$, $\b^{(\nu)}$, and $\g^{(\nu)}$ in
$L_2(\M)$ such that:
\begin{equation}\label{adapted}
d_n(u_\nu)= \a_n^{(\nu)} + \b_n^{(\nu)} + \g_n^{(\nu)}  \  \text{for all}\   n \geq 1,
\end{equation}
\begin{equation}\label{J-a}
J\big(\sum_{n\geq 1}\a_n^{(\nu)} \otimes e_{n}, t\big) \leq \kappa J(u_\nu,t),\  
t>0,
\end{equation}
\begin{equation}\label{J-b}
 J\big(\big(\sum_{n\geq 1} \E_{n-1}(|\b_n^{(\nu)}|^2) \big)^{1/2}, t\big) \leq \kappa J(u_\nu,t), \ t>0,
\end{equation}
\begin{equation}\label{J-c}
J\big(\big(\sum_{n\geq 1} \E_{n-1}(|{\g_n^{(\nu)}}^*  |^2) \big)^{1/2}, t\big) \leq \kappa J(u_\nu,t), \ t>0,
\end{equation}
where the $J$-functional in the left hand side of the inequality
in \eqref{J-a}  is taken relative to the interpolation couple
$(L_{1,\infty}(\M \overline{\otimes} \ell_\infty), L_2(\M
\overline{\otimes} \ell_\infty))$ and those from the
left hand sides of \eqref{J-b}  and \eqref{J-c} are taken with respect to the
interpolation couple  $(L_{1,\infty}(\M), L_2(\M))$. We set
\[
\a_n= \sum_{\nu \in \mathbb{Z}} \a_n^{(\nu)}, \  \b_n= \sum_{\nu \in \mathbb{Z}} \b_n^{(\nu)}, \  \text{and}\  \g_n= \sum_{\nu \in \mathbb{Z}} \g_n^{(\nu)}.
\]
Then we obtain  three adapted sequence $\a=(\a_n)_n$, $\b=(\b_n)_n$, and $\g=(\g_n)_n$. Define the martingale difference sequences
\[
da_n=\a_n -\E_{n-1}(\a_n), \  db_n=\b_n -\E_{n-1}(\b_n),\  \text{and}\  dc_n=\g_n -\E_{n-1}(\g_n).
\]
We claim that the resulting martingales $a$, $b$, and $c$ satisfy the conclusion of the theorem. Indeed, it is clear from the construction that $x=a +b+c$.  For the second item, we will verify  the statement  separately for $a$, $b$, and $c$. We begin with the martingale $a$.  This will be deduced from the next lemma. For $0<\theta<1$, $1<p<2$, and an interpolation couple $(X_0,X_1)$,  $(X_0, X_1)_{\theta, p, \underline{J}}$ and $(X_0, X_1)_{\theta, p, \underline{K}}$ denote the discrete real interpolation methods  using the $J$-functionals and $K$-functionals, respectively. We refer to  \cite{BL} for definitions.
\begin{lemma}\label{lemma1} For every $0<\theta<1$ and  every $1<p<2$,
 \begin{equation*}
\Big\| ( \a_n)_{n\geq 1} \Big\|_{[L_{1,\infty}(\M \overline{\otimes} \ell_\infty) , L_2(\M
\overline{\otimes} \ell_\infty)]_{\theta,p;\underline{J}}} \leq
4\kappa\|x\|_{[L_1(\M), L_2(\M)]_{\theta,p;\underline{K}}}.
\end{equation*}
\end{lemma}
For $\nu \in \mathbb{Z}$, let $[\alpha]^{(\nu)}=\sum_n \alpha_n^{(\nu)} \otimes e_n$. The series $\sum_\nu [\a]^{(\nu)}$ is a representation of $\sum_n \alpha_n \otimes e_n$. Then the lemma follows immediately from combining \eqref{fundamental} and \eqref{J-a}.

Fix $1<p< 2$ and $1/p=(1-\theta) + \theta/2$. We appeal to  the  known facts  from \cite{PX3} that for any semifinite von Neumann algebra $\N$, we have
\[
L_p(\N)=\big[ L_{1,\infty}(\N), L_2(\N)\big]_{\theta, p, \underline{J}} \ \text{and}\  L_p(\N)=\big[ L_{1}(\N), L_2(\N)\big]_{\theta, p, \underline{K}}.\]
 The above lemma yields a constant $c_p$ such that 
\begin{equation*}
\Big( \sum_n \|\a_n\|_p^p\Big)^{1/p} =\Big\|  (\a_n)_{n\geq 1} \Big\|_{L_p(\M \overline{\otimes} \ell_\infty)} \leq
c_p\big\|x\big\|_{p}.
\end{equation*}
Applying the fact that  conditional expectations are contractive projections in $L_p(\M)$ 
 gives 
 \begin{equation*}
 \big\| a \big\|_{\h_p^d} \leq 2c_p \big\|x\big\|_p.
 \end{equation*}
  Now we sketch the argument for $b$.
 We consider the conditioned spaces involved as
 subspaces of  $L_r$-spaces associated to $\M \overline{\otimes}
B(\ell_2(\mathbb{N}^2))$  for appropriate values of $r$. Then for every
$\nu \in  \Z$,
\begin{equation*}
J\big(\big(\sum_{n\geq 1} \E_{n-1}(|{\b_n^{(\nu)}}
|^2)\big)^{1/2}, 2^\nu\big) =J\left(\b^{(\nu)}, 2^{\nu};
L_{1,\infty}(\M\overline{\otimes} B(\ell_2(\mathbb{N}^2))),
L_{2}(\M\overline{\otimes} B(\ell_2(\mathbb{N}^2)))\right).
\end{equation*}
Therefore,  \eqref{J-b} becomes,
\begin{equation}
 J\left(\b^{(\nu)}, 2^{\nu};
L_{1,\infty}(\M\overline{\otimes} B(\ell_2(\mathbb{N}^2))),
L_{2}(\M\overline{\otimes} B(\ell_2(\mathbb{N}^2)))\right)
\leq \kappa J(u_\nu, 2^\nu).
\end{equation}
Using similar argument as in the estimate of  norm of $\alpha$  with $\M \overline{\otimes}\ell_\infty$ replaced by $\M\overline{\otimes} B(\ell^2(\mathbb{N}^2))$, we get  as in Lemma~\ref{lemma1} that for every $0<\theta<1$ and $1<p<2$,
\begin{equation*}
\big\| \b\big\|_{[L_{1,\infty}(\M \overline{\otimes}
B(\ell_2(\mathbb{N}^2))) , L_{2}(\M \overline{\otimes}
B(\ell_2(\mathbb{N}^2)))]_{\theta,p;\underline{J}}} \leq
4\kappa\big\|x\big\|_{[L_1(\M), L_2(\M)]_{\theta,p;\underline{K}}} .
\end{equation*}
As in the previous case, applying real interpolations with appropriate values of $\theta$ and $p$ gives that for every $1<p<2$,
\begin{equation*}
\big\| \b\big\|_{L_p(\M \overline{\otimes}
B(\ell_2(\mathbb{N}^2)))}  \leq
c_p \big\| x \big\|_{p}.
\end{equation*}
This is equivalent to 
\begin{equation*}
\Big\| \big( \sum_{n\geq 1} \E_{n-1}(|\b_n|^2) \big)^{1/2}\Big\|_p  \leq
c_p \big\| x \big\|_{p}.
\end{equation*}
Using Kadison's inequality $\E_{n-1}(\b_n)^*\E_{n-1}(\b_n) \leq \E_{n-1}|\b_n|^2$ for all $n\geq 1$, we deduce that 
\begin{equation*}
\big\| b \big\|_{\h_p^c}  \leq
2c_p \big\| x \big\|_{p}.
\end{equation*}
Similar argument  can be applied to the sequence $\gamma$  to deduce the corresponding  estimate
\begin{equation*}
\big\| c \big\|_{\h_p^r}  \leq
2c_p \big\| x \big\|_{p}.
\end{equation*}
Combining the above three estimates clearly provides the second item in the statement of the theorem. This completes the proof for the finite case.

  \smallskip
 
 \noindent {\bf Case 2.}  Assume now that $\M$ is infinite.  Since $\M_*$ is separable,  the von Neumann algebra $\M$ is $\sigma$-finite.  We note first that Case~1 extends easily to any finite case with the trace $\T$  being not  necessarily normalized (with the same constants as  in the case of normalized trace).  Since there is a  trace preserving conditional expectation $\E_1: \M \to \M_1$, it is known that $\T|_{\M_1}$ remains semifinite.
 
 Fix an increasing  sequence of projections $(e_k)_{k\geq 1} \subset \M_1$  with $\T(e_k)<\infty$ for all $k\geq 1$ and such that $(e_k)_{k\geq 1}$ converges to ${\bf 1}$ for the strong operator topology. For each $k$, consider the finite von Neumann algebra 
  $(e_k \M e_k, \T|_{e_k\M e_k})$ with the
 filtration  $(e_k\M_n e_k)_{n\geq 1}$.   If we denote by $\E_n^{(k)}$ the trace preserving conditional expectation from $e_k\M e_k$ onto $e_k \M_n e_k$ then $\E_n^{(k)}$ is  just the restriction of $\E_n$ on $e_k \M e_k$.  This is the case since the $e_k$'s were chosen from the smallest subalgebra $\M_1$.
 Therefore, if $y \in e_k\M e_k$ then one can easily verify that
 \[
 \big\|y \big\|_{\h_p^d(e_k\M e_k)} = \big\|y \big\|_{\h_p^d(\M)}, \ \big\|y \big\|_{\h_p^c(e_k\M e_k)} = \big\|y \big\|_{\h_p^c(\M)}, \ \text{and} \ 
 \big\|y \big\|_{\h_p^r(e_k\M e_k)} = \big\|y \big\|_{\h_p^r(\M)}.
\]
 
Let $x \in L_2(\M)\cap L_1(\M)$. For each $k\geq 1$, $e_kxe_k \in L_2(e_k\M e_k)$.   From Case 1., there exists a decomposition
$
e_kxe_k= a^{(k)} + b^{(k)} + c^{(k)}$
  with the property that for  every $1<p <2$,
\[
\big\| a^{(k )}\big\|_{\h^d_p}
 + \big\| b^{(k)} \big\|_{\h_p^c} + \big\| c^{(k)}
\big\|_{\h_p^r}\leq \kappa_p \big\| e_kxe_k \big\|_p.
\]
Fix an ultrafilter $\mathcal{U}$ on   $\mathbb{N}$ containing the Fr\'echet filter. For any given   $1<p<2$, the weak-limit along the ultrafilter $\mathcal{U}$ of  the sequence $(a^{(k)})_{k\geq 1}$ exists in $\h_p^d(\M)$. It is crucial here to observe that   such weak-limits are independent of $p$ (since they are automatically weak-limits  of the same sequence in $L_1(\M\overline{\otimes}\ell_\infty) + L_2(\M\overline{\otimes}\ell_\infty)$). Similar observations can be made with the sequences $(b^{(k)})_{k\geq 1}$ and $(c^{(k)})_{k\geq 1}$.
Set 
\[
a= \text{w-}\lim_{k,\mathcal U} a^{(k)}, \ \  b =\text{ w-}\lim_{k,\mathcal U} b^{(k)},  \ \ \text{and}   \ \ c=\text{ w-}\lim_{k,\mathcal U} c^{(k)}
\]
 in $\h_p^d(\M)$, $\h_p^c(\M)$, and $\h_p^r(\M)$, respectively. We also observe  that  for every $1<p<2$,  it is easy to verify that $\lim_{k\to \infty}\|e_k xe_k -x\|_p=0$.  A fortiori, $\lim_{k, \mathcal{U}}\|e_k xe_k -x\|_p=0$.  All these facts lead to the decomposition:
\[
x=a +b +c.
\]
 Furthermore, for every $1<p<2$, we have
\begin{align*}
\big\| a \big\|_{\h^d_p}
 + \big\| b\big\|_{\h_p^c} + \big\| c
\big\|_{\h_p^r} &\leq \sup_k\Big\{\big\| a^{(k)}\big\|_{\h^d_p}
 + \big\| b^{(k)} \big\|_{\h_p^c} + \big\| c^{(k)}
\big\|_{\h_p^r}\Big\}\\
&\leq   \kappa_p \sup_k \big\| e_k xe_k \big\|_p  \leq 
 \kappa_p \big\|x\big\|_p
\end{align*}
where $\kappa_p$ is the constant from Case~1.
The proof is complete.
\end{proof}

\begin{remarks}  1)  Since the noncommutative Burkholder inequalities do not hold  for $p=1$, the validity of our simultaneous decomposition can not include the left endpoint of the interval $(1,2)$. On the other hand,  using known estimates from real interpolation $(\theta, p,K)$ and $(\theta, p,J)$ methods  of classical  Lebesgue spaces, we can derive  that  there is an absolute constant $C$ such that  for $1/p= (1-\theta) + \theta/2$, we have $\kappa_p \leq C\theta^{-2} (1-\theta)^{-1/2 -1/p}$. It follows that  
$\kappa_p$ is of order  $(p-1)^{-2}$ when $p\to 1$ and of order $(2-p)^{-1}$ when $p \to 2$. In particular, our method of proof does not allow  any extension of the decomposition to any of the  endpoints of the interval $[1,2]$. As we only get that $\kappa_p=O((p-1)^{-2})$ when $p\to 1$, our arguments  do not 
yield the optimal order for the constants for the noncommutative  Burkholder inequalities from \cite{Ran21}.

2) Junge and Perrin also considered simultaneous type decompositions  for conditioned Hardy spaces in \cite{Junge-Perrin}.
Our Theorem~\ref{simultaneous}  above   should be compared with  \cite[Theorem~5.9]{Junge-Perrin}.  See also Corollary~\ref{simultaneous2}  below for similar type simultaneous decompositions for  the case of martingale Hardy space norms.

\end{remarks}


\section{Burkholder's  inequalities in symmetric spaces}

The following is the principal result of this article. It provides  extensions of noncommutative Burkholder's inequalities for martingales in general noncommutative symmetric spaces.

\begin{theorem}\label{main}  Let $E$ be a symmetric Banach function space  on $(0,\infty)$ satisfying the Fatou property. Assume that  either $1<p_E \leq q_E < 2$ or $2<p_E \leq q_E<\infty$. Then 
\[
E(\M)= \h_E(\M).
\]
That is, a martingale $x=(x_n)_{n\geq 1}$ is bounded in $E(\M)$ if and only if  it belongs to $\h_E(\M)$ and
\[
\big\|x\big\|_{E(\M)} \simeq_E \big\|x\big\|_{\h_E}.
\]
\end{theorem}

As noted in the introduction, the preceding theorem  solves  positively a question raised in \cite{Jiao2}. The new result here is the case  where $1<p_E\leq q_E<2$.
The case $2< p_E \leq q_E <\infty$ was established by Dirksen in \cite[Theorem~6.2]{Dirksen2} but we will also provide  an alternative approach  for this range. We remark that under the assumptions of Theorem~\ref{main},  the Banach function space $E$  is fully symmetric in the sense of  \cite{DDP4} but this extra property will not  be needed  in the proof.

\medskip

We divide the proof into four  separate parts according to $1<p_E \leq q_E < 2$ or $2<p_E\leq q_E<\infty$, each case involving two inequalities.  The main difficulty  in the proof  is Part~II below. Part ~III  will be  deduced from Part~II via duality. The other two parts will be derived   from standard use of interpolations of linear operators. 

\subsection{The case ${ 1<p_E \leq q_E <2}$} Let $E$ be a  symmetric Banach function space on $(0,\infty)$ with the Fatou property and  satisfying $1<p_E\leq q_E<2$.  Throughout the proof, we fix $p$ and $q$ so that $1<p<p_E\leq q_E<q<2$. In this case, $E \in {\rm Int}(L_p,L_q)$.

\bigskip

\noindent{\bf Part~I.} 
We will verify  that there exists a constant $c_E$ such that for every $x \in \h_E(\M)$,  we have
\[
\big\|x\big\|_{E(\M)} \leq c_E \big\|x\big\|_{\h_E}.
\]
This is a simple consequence of the noncommutative Burkholder inequalities and 
Proposition~\ref{comp-Int}. Indeed, for $1<r<2$, $\big\|x\big\|_r \leq c_r\big\|x\big\|_{\h_r^d}$. By interpolation, we deduce that $\big\|x\big\|_{E(\M)} \leq c_E\big\|x\big\|_{\h_E^d(\M)}$. Similar arguments also give $\big\|x\big\|_{E(\M)} \leq c_E\big\|x\big\|_{\h_E^c(\M)}$ and $\big\|x\big\|_{E(\M)} \leq c_E\big\|x\big\|_{\h_E^r(\M)}$.

 Now, assume that  $x\in \h_E(\M)$ is such that $x =w +y +z$ where $w \in \h_E^d(\M)$, $y \in \h_E^r(\M)$, and $z \in \h_E^r(\M)$.  Then  we have
\[
\big\|x\big\|_{E(\M)} \leq c_E \big( \big\|w\big\|_{\h_E^d} + \big\|y\big\|_{\h_E^c} + \big\|z\big\|_{\h_E^r}\big).
\]
 Taking the infimum over all decompositions $x=w +y +z$ provides the desired inequality.
\qed

\medskip

We should emphasize that   the inequality from Part~I  may be interpreted as inclusion mappings  in the following sense:   if $w\in\{d,c,r\}$,  then there is a natural  bounded  map $\xi_E^w: \h_E^w(\M) \to E(\M)$.  In fact,  $\xi_E^w$  may be taken as the bounded extension of the map $(\mathcal{F}_M, \|\cdot\|_{\h_E^w}) \to  E(\M)$  defined by $(x_n)_{n\geq 1} \mapsto \lim_{n\to \infty} x_n$ (since  $(x_n)$ is a finite sequence, the limit should be  understood  as the final value  of $(x_n)$).
We claim that these maps are one to one. To verify this claim, we consider the following diagram:
  \[
\begin{CD}
\h_E^w(\M) @>{\xi_E^w}>>  E(\M)\\
@V{\iota}VV @VV{j}V\\
\h_p^w(\M) +\h_q^w(\M)  @>{\xi_{p,q}^w}>>  L_p(\M) + L_q(\M),
\end{CD}
\]
where $\iota$ is the inclusion map from  the interpolation in Proposition~\ref{comp-Int}, $j$ is the formal inclusion, and $\xi_{p,q}^w$ is the combination of $\xi_p^w$ and 
$\xi_q^w$ introduced in the previous section.  When $z\in \mathcal{F}_M$, we clearly have $\xi_{p,q}^w \iota(z)=j \xi_E^w(z)$. Therefore,  the above diagram commutes. Since $\iota$ and $\xi_{p,q}^w$ are one to one, so is $\xi_E^w$. As a consequence, any element  $a \in \h_E^w(\M)$  is   uniquely associated  with the operator  $\xi_E^w(a) \in E(\M)$ which we will still denote by $a$. In Part~II below, statement   such as $x=a +b + c$ for $x\in E(\M)$, $ a \in \h_E^d(\M)$, $ b \in \h_E^c(\M)$, and $ c \in \h_E^r(\M)$ should be understood to mean $x=\xi_E^d(a) +\xi_E^c(b) + \xi_E^r(c)$.

\bigskip

\noindent{\bf Part~II.}
We consider now the reverse inequalities. That is,  there exists a constant $\beta_E$ such that for every $x \in E(\M)$,
\[
\big\|x\big\|_{\h_E} \leq \beta_E \big\|x\big\|_{E(\M)}.
\]
The  proof  is much more involved and requires several steps.
 Our approach relies on  two essential facts.   As  stated in Corollary~\ref{lifted-interpolation},   noncommutative symmetric spaces  have concrete representations as  interpolation spaces. The second fact is  the simultaneous decomposition obtained in the previous section.
 
According to Corollary~\ref{lifted-interpolation}, we may  fix  a  symmetric Banach function space $F$ on $(0,\infty)$ with nontrivial Boyd indices and such that for any semifinite von Neumann algebra $\mathcal{N}$, we have:
\[
E(\mathcal{N}) = \big[ L_p(\mathcal{N}), L_q(\mathcal{N})\big]_{F, \underline{j}},
\]
where $[\cdot, \cdot]_{F,\underline{j}}$ is the interpolation method introduced in the previous section.
We begin with the following intermediate result.

\begin{lemma}\label{approximation} Let $x \in L_p(\M) \cap L_q(\M)$. For every $\epsilon>0$, there exist $x^{(\epsilon)} \in  L_p(\M) \cap L_q(\M)$,  martingales $a^{(\epsilon)} \in \h_E^d(\M)$, $b^{(\epsilon)} \in \h_E^c(\M)$, and $c^{(\epsilon)} \in \h_E^r(\M)$  with:
\begin{enumerate}[{\rm(1)}]
\item $\displaystyle{\big\| x-x^{(\epsilon)}\big\|_{L_p(\M) \cap L_q(\M)} <\epsilon}$;
\item $x^{(\epsilon)} =a^{(\epsilon)} + b^{(\epsilon)}+ c^{(\epsilon)}$;
\item $\displaystyle{\big\|a^{(\epsilon)}\big\|_{\h_E^d} +\big\|b^{(\epsilon)}\big\|_{\h_E^c} +\big\|c^{(\epsilon)}\big\|_{\h_E^r} \leq \eta_E \big\|x\big\|_{E(\M)}}$.
\end{enumerate}
\end{lemma}

\begin{proof}
 Let $x \in L_p(\M) \cap L_q(\M)$ and $\epsilon>0$.  Using  the interpolation couple $\big(L_p(\M), L_q(\M)\big)$,
 fix a representation $x=\sum_{\nu \in \mathbb{Z}} u_\nu$
 (convergent in $L_p(\M)+L_q(\M)$) such that
 \begin{equation}\label{rep-norm}
 \Big\| \underline{j}\big(\{u_\nu\}_\nu, \cdot\big)\Big\|_F \leq 2 \big\|x \big\|_{F, \underline{j}}.
 \end{equation}
Note that the $u_\nu$'s belong to $L_p(\M) \cap L_q(\M)$. Using Lemma~\ref{approximation2},
for each $\nu \in \mathbb{Z}$, we may choose  $u_\nu^{(m_\nu)} \in L_1(\M) \cap L_2(\M)$ satisfying the following properties:
\begin{enumerate}
\item $\displaystyle{\big\| u_\nu^{(m_\nu)} -u_\nu \big\|_{L_p(\M) \cap L_q(\M)} \leq  \frac{\epsilon}{4^{|\nu| +1}}}$;
\item $\displaystyle{\big\| u_\nu^{(m_\nu)} \big\|_q \leq \big\| u_\nu \big\|_q}$;
\item $\displaystyle{\big\| u_\nu^{(m_\nu)} \big\|_p \leq \big\| u_\nu \big\|_p}$.
\end{enumerate}
The last two conditions imply that for every  $\nu \in \mathbb{Z}$ and  every $t>0$,
\begin{equation*}
J\big(u_\nu^{(m_\nu)}, t\big) \leq J\big(u_\nu, t \big),
\end{equation*}
which furthermore leads to the following inequality: 
\begin{equation}\label{j-inequality}
\underline{j}\big(\{u_\nu^{(m_\nu)}\}_\nu,\cdot\big) \leq \underline{j}\big(\{u_\nu\}_\nu, \cdot\big).
\end{equation} 
We define the operator $x^{(\epsilon)}$ by setting:
\[
x^{(\epsilon)}=\sum_{\nu \in \mathbb{Z}}  u_\nu^{(m_\nu)}.
\]
Then it satisfies the following  norm estimates:
\begin{align*}
\big\| x^{(\epsilon)} -x\big\|_{L_p(\M)  \cap L_q(\M)} &\leq  \sum_{\nu \in \mathbb{Z}} \big\|u_\nu^{(m_\nu)}-u_\nu \big\|_{L_p(\M) \cap L_q(\M)}\\
&\leq \sum_{\nu \in \mathbb{Z}} \frac{\epsilon}{4^{|\nu| +1}} \leq \epsilon.
\end{align*}
In particular, $x^{(\epsilon)} \in L_p(\M) \cap L_q(\M)$ and the first item in the statement of Lemma~\ref{approximation} is satisfied. The crucial fact here is that all  $u_\nu^{(m_\nu)}$'s in the representation of $x^{(\epsilon)}$  belong to $L_1(\M) \cap L_2(\M)$ so that Theorem~\ref{simultaneous} can be applied to each of the $u_\nu^{(m_\nu)}$'s. That is, for every $\nu \in \mathbb{Z}$, there exist $a_\nu \in \cap_{1<s<2} \h_s^d(\M)$, $b_\nu \in \cap_{1<s<2} \h_s^c(\M)$, and $c_\nu \in \cap_{1<s<2} \h_s^r(\M)$   satisfying:
\begin{equation}
u_\nu^{(m_\nu)}=a_\nu +b_\nu + c_\nu
\end{equation}
and if $s\in \{p,q\}$, then
\begin{equation}\label{double-decomposition}
\big\|a_\nu \big\|_{\h_s^d} + \big\|b_\nu \big\|_{\h_s^c} +\big\|c_\nu \big\|_{\h_s^r} \leq 
\kappa(p,q) \big\| u_\nu^{(m_\nu)} \big\|_s
\end{equation}
where $\kappa(p,q)=\max\{\kappa_p, \kappa_q\}$.

For each $\nu \in \mathbb{Z}$,  we consider $\mathcal{D}_d(a_\nu)\in L_p(\M \overline{\otimes}\ell_\infty) \cap L_q(\M \overline{\otimes}\ell_\infty)$,  
$U\mathcal{D}_c(b_\nu) \in L_p(\M \overline{\otimes} B(\ell_2(\mathbb{N}^2))) \cap L_q(\M \overline{\otimes} B(\ell_2(\mathbb{N}^2)))$, and
$U\mathcal{D}_c(c_\nu^*) \in L_p(\M \overline{\otimes} B(\ell_2(\mathbb{N}^2))) \cap L_q(\M \overline{\otimes} B(\ell_2(\mathbb{N}^2)))$.

First, we observe that \eqref{double-decomposition} can be reinterpreted using the $J$-functionals as follows:

\begin{equation}\label{J-a1}
J\big( \mathcal{D}_d(a_\nu), t) \leq \kappa(p,q) J(u_\nu^{(m_\nu)},t),\  
t>0,
\end{equation}
\begin{equation}\label{J-b1}
 J\big(U\mathcal{D}_c(b_\nu), t\big) \leq \kappa(p,q) J(u_\nu^{(m_\nu)},t), \ t>0,
\end{equation}
\begin{equation}\label{J-c1}
J\big(U\mathcal{D}_c(c_\nu^*), t\big) \leq \kappa(p,q) J(u_\nu^{(m_\nu)},t), \ t>0
\end{equation}
where the $J$-functional  on the left side of  \eqref{J-a1}  is taken using the couple
$[L_p(\M \overline{\otimes}\ell_\infty),L_q(\M \overline{\otimes}\ell_\infty)]$ and those on the left sides   of inequalities \eqref{J-b1} and \eqref{J-c1} were computed using the couple $[ L_p(\M \overline{\otimes} B(\ell_2(\mathbb{N}^2))), L_q(\M \overline{\otimes} B(\ell_2(\mathbb{N}^2)))]$.

We need the following properties of the  sequences $\{\mathcal{D}_d(a_\nu)\}_{\nu \in \mathbb{Z}}$, 
$\{U\mathcal{D}_c(b_\nu)\}_{\nu \in \mathbb{Z}}$, and $\{U\mathcal{D}_c(c_\nu^*)\}_{\nu \in \mathbb{Z}}$. We refer to \cite{D1} for definition and criterion for unconditionally Cauchy series in Banach spaces.
\begin{sublemma}\label{sublemma} 
\begin{enumerate}[{\rm(1)}]
 \item  $ \sum_{\nu \in \mathbb{Z}} \mathcal{D}_d(a_\nu)$  is  a weakly unconditionally Cauchy series  in  $E(\M \overline{\otimes} \ell_\infty)$. 
\item $ \sum_{\nu \in \mathbb{Z}} U\mathcal{D}_c(b_\nu)$  and $\sum_{\nu \in \mathbb{Z}} U\mathcal{D}_c(c_\nu^*)$ are weakly unconditionally Cauchy   series in $E(\M\overline{\otimes} B(\ell_2(\mathbb{N}^2)))$. 
\end{enumerate}
Moreover, there exists a constant $\kappa_E$ such that:
\begin{equation*}
\begin{split}
\max\Big\{\sup_{N\geq 1}\|S_N(a)\|_{E(\M \overline{\otimes} \ell_\infty)}, \sup_{N\geq 1}\|S_N(b)\|_{E(\M \overline{\otimes} B(\ell_2(\mathbb{N}^2)))}, &\sup_{N\geq 1}\|S_N(c^*)\|_{E(\M \overline{\otimes} B(\ell_2(\mathbb{N}^2)))} \Big\}\\
&\leq  \kappa_E\|x\|_{E(\M)},
\end{split}
\end{equation*}
where for each $N\geq 1$, $S_N(a)=\sum_{|\nu|\leq N} \mathcal{D}_d(a_\nu)$, $S_N(b)=\sum_{|\nu|\leq N} U\mathcal{D}_c(b_\nu)$, and $S_N(c^*)=\sum_{|\nu|\leq N} U\mathcal{D}_c(c_\nu^*)$.
\end{sublemma}

To verify the first item in Sublemma~\ref{sublemma},  we note  that if $S$ if a finite subset of $\mathbb{Z}$ then  it follows from \eqref{j-inequality} and  \eqref{J-a1} that 
\[
\underline{j}\big(\{ \mathcal{D}_d(a_\nu)\}_{\nu \in S}, \cdot\big) \leq \kappa(p,q)\underline{j}\big(\{u_\nu^{(m_\nu)}\}_{\nu \in S}, \cdot\big) \leq \kappa(p,q)\underline{j}\big(\{u_\nu\}_{\nu}, \cdot\big).
\]
By the definition of $[\cdot, \cdot]_{F, \underline{j}}$, for every finite sequence of scalars $(\theta_\nu)_{\nu\in S}$ with $|\theta_\nu|=1$, we have 
\[
\big\| \sum_{\nu \in S} \theta_\nu \mathcal{D}_d(a_\nu) \big\|_{[L_p(\M \overline{\otimes} \ell_\infty),L_q(\M \overline{\otimes} \ell_\infty)]_{F, \underline{j}}} \leq \kappa(p,q)  \big\| \underline{j}\big(\{u_\nu\}_{\nu}, \cdot\big)\|_F \leq 2\kappa(p,q) \big\|x \big\|_{F,\underline{j}}
\]
where the last inequality is from \eqref{rep-norm}. Now we use the facts that 
\[
E(\M \overline{\otimes} \ell_\infty) =\big[L_p(\M \overline{\otimes} \ell_\infty), L_q(\M \overline{\otimes} \ell_\infty)\big]_{F, \underline{j}} \ \text{and} \ 
E(\M) =\big[L_p(\M), L_q(\M)\big]_{F, \underline{j}}
\]
to  deduce that there exists a constant  $\kappa_E$ such that: 
\begin{equation}\label{Cauchy-a}
\big\| \sum_{\nu \in S} \theta_\nu \mathcal{D}_d(a_\nu) \big\|_{E(\M \overline{\otimes} \ell_\infty)} \leq \kappa_E \big\|x \big\|_{E(\M)}.
\end{equation}
Since  this is the case for any arbitrary finite subset of $\mathbb{Z}$, it proves that
 the series $\sum_{\nu \in \mathbb{Z}} \mathcal{D}_d(a_\nu)$ is weakly unconditionally Cauchy in $E(\M \overline{\otimes} \ell_\infty)$. 
 
The proof of the second item follows the same pattern. As above, if $S$ is a finite subset  of $\mathbb{Z}$, then it follows from \eqref{j-inequality} and \eqref{J-b1} that:
\[
\underline{j}\big(\{U\mathcal{D}_c(b_\nu)\}_{\nu \in S}, \cdot\big) \leq \kappa(p,q)\underline{j}\big(\{u_\nu^{(m_\nu)}\}_{\nu \in S}, \cdot\big) \leq \kappa(p,q)\underline{j}\big(\{u_\nu\}_{\nu}, \cdot\big).
\]
Using similar arguments as above, we may deduce that for every finite subset $S$ of $\mathbb{Z}$ and for every sequence of scalars $(\theta_\nu)_{\nu\in S}$ with $|\theta_\nu|=1$,
\begin{equation}\label{Cauchy-b}
\big\| \sum_{\nu \in S} \theta_\nu U\mathcal{D}_c(b_\nu) \big\|_{E(\M \overline{\otimes} B(\ell_2(\mathbb{N}^2))) } \leq \kappa_E\big\|x \big\|_{E(\M)}. 
\end{equation}
This again shows that  the series $\sum_{\nu \in \mathbb{Z} } U\mathcal{D}_c(b_\nu) $ is  weakly unconditionally Cauchy in $E(\M \overline{\otimes} B(\ell_2(\mathbb{N}^2)))$. The proof for the series $\sum_{\nu \in \mathbb{Z}} U\mathcal{D}_c(c_\nu^*)$ is identical so we omit the details. The inequality stated in  Sublemma~\ref{sublemma} follows from \eqref{Cauchy-a}, \eqref{Cauchy-b}, and the corresponding inequality for  $\sum_{\nu \in \mathbb{Z}} U\mathcal{D}_c(c_\nu^*)$.  Sublemma~\ref{sublemma} is verified.

\medskip

Next, we note that since $L_p(\M \overline{\otimes} \ell_\infty) + L_q(\M \overline{\otimes} \ell_\infty)$ is a reflexive space, the series  $\sum_{\nu \in \mathbb{Z}} \mathcal{D}_d(a_\nu)$ is unconditionally convergent in $L_p(\M \overline{\otimes} \ell_\infty) + L_q(\M \overline{\otimes} \ell_\infty)$. Similarly, both  series $\sum_{\nu \in \mathbb{Z}} U\mathcal{D}_c(b_\nu)$ and $\sum_{\nu \in \mathbb{Z}} U\mathcal{D}_c(c_\nu^*)$  are unconditionally convergent in
$L_p(\M \overline{\otimes} B(\ell_2(\mathbb{N}^2))) + L_q(\M \overline{\otimes} B(\ell_2(\mathbb{N}^2)))$.
Now we set:
\[
  \displaystyle{\alpha^{(\epsilon)} := \sum_{\nu \in \mathbb{Z}} \mathcal{D}_d(a_\nu}) \in L_p(\M \overline{\otimes} \ell_\infty) + L_q(\M \overline{\otimes} \ell_\infty),
\]    
\[
\displaystyle{\beta^{(\epsilon)} := \sum_{\nu \in \mathbb{Z}} U\mathcal{D}_c(b_\nu}) \in 
L_p(\M \overline{\otimes} B(\ell_2(\mathbb{N}^2))) + L_q(\M \overline{\otimes} B(\ell_2(\mathbb{N}^2))),
\]
and
\[ \displaystyle{\gamma^{(\epsilon)} := \sum_{\nu \in \mathbb{Z}} U\mathcal{D}_c(c_\nu^*)} \in L_p(\M \overline{\otimes} B(\ell_2(\mathbb{N}^2))) + L_q(\M \overline{\otimes} B(\ell_2(\mathbb{N}^2))).
\] 
We claim that
\begin{equation}\label{d-inequality}
\max\Big\{ \big\|  \alpha^{(\epsilon)} \big\|_{E(\M \overline{\otimes} \ell_\infty)},
\big\| \beta^{(\epsilon)}\big\|_{E(\M \overline{\otimes} B(\ell_2(\mathbb{N}^2)))}, \big\|\gamma^{(\epsilon)}\big\|_{E(\M \overline{\otimes} B(\ell_2(\mathbb{N}^2)))}\Big\}\leq \kappa_E\big\|x\big\|_{E(\M)}.
\end{equation} 
To verify this claim, we use the fact  mentioned in the previous section that for every semifinite von Neumann algebra $\N$, the inclusion map  from $ E(\N)$ into  $L_p(\N) + L_q(\N)$ is a semi-embedding. Indeed, if $\rho=\kappa_E\|x\|_{E(\M)}$, then from Sublemma~\ref{sublemma},  $(S_N(a))_{N\geq 1}$ is a sequence in the $\rho$-ball of 
 $E(\M \overline{\otimes} \ell_\infty)$  that converges to $\alpha^{(\epsilon)}$ for the norm topology of 
 $L_p(\M \overline{\otimes} \ell_\infty) +L_q(\M \overline{\otimes} \ell_\infty)$. By semi-embedding, we have  $\alpha^{(\epsilon)} \in E(\M \overline{\otimes} \ell_\infty)$ with 
 $\|\alpha^{(\epsilon)}\|_{E(\M \overline{\otimes} \ell_\infty)} \leq \rho$. Identical  arguments  can be applied to $\beta^{(\epsilon)}$ and $\gamma^{(\epsilon)}$ to deduce that 
 $\big\| \beta^{(\epsilon)}\big\|_{E(\M \overline{\otimes} B(\ell_2(\mathbb{N}^2)))}\leq \rho$ and
 $\big\| \gamma^{(\epsilon)}\big\|_{E(\M \overline{\otimes} B(\ell_2(\mathbb{N}^2)))} \leq \rho$. We have verified \eqref{d-inequality}.
 
 We observe that since the sequence $(S_N(a))_N$ is from $\mathcal{D}_d(\h_p^d(\M)) \cap \mathcal{D}_d(\h_q^d(\M))$, it follows that  $\alpha^{(\epsilon)} \in \mathcal{D}_d(\h_{L_p +L_q}^d(\M))$. That is, there exists  $a^{(\epsilon)} \in \h_{L_p +L_q}^d(\M)$ such that $\alpha^{(\epsilon)}=\mathcal{D}_d(a^{(\epsilon)})$. But since $ \big\|  \alpha^{(\epsilon)} \big\|_{E(\M \overline{\otimes} \ell_\infty)}\leq \kappa_E \big\|x\big\|_{E(\M)}$, by  the complementation stated in Proposition~\ref{comp-Int}, we conclude that  $a^{(\epsilon)} \in \h_E^d(\M)$ with the norm estimate:
 \[
 \big\| a^{(\epsilon)}\big\|_{h_E^d} \leq \kappa_E \big\|x\big\|_{E(\M)}.
 \]
 Identical arguments can be applied to the sequences $(S_N(b))_N$ and $(S_N(c))_N$ to  deduce that there exist martingales $b^{(\epsilon)} \in \h_E^c(\M)$ and $c^{(\epsilon)} \in \h_E^r(\M)$ such that $\beta^{(\epsilon)}=U\mathcal{D}_c(b^{(\epsilon)})$, $\gamma^{(\epsilon)}=U\mathcal{D}_c((c^{(\epsilon)})^*)$, and
 \[
 \max\Big\{\big\|b^{(\epsilon)}\big\|_{\h_E^c}, \big\|c^{(\epsilon)}\big\|_{\h_E^r}\Big\}\leq 
\kappa_E \big\|x\big\|_{E(\M)}. 
 \]
 It is clear from the construction that 
$x^{(\epsilon)}= a^{(\epsilon)} + b^{(\epsilon)} + c^{(\epsilon)}$ and 
the last two inequalities  clearly implies the last item in  Lemma~\ref{approximation}. The proof  is complete.
\end{proof}

The next step provides the desired  decomposition for all $x \in L_p(\M) \cap L_q(\M)$.
\begin{lemma}\label{approximation-2} There exists a constant $\beta_E$ such that  every 
 $ x \in  L_p(\M) \cap L_q(\M)$ admits  a decomposition $x=a+b+c$  where $a\in \h_E^d(\M)$, $b \in \h_E^c(\M)$, and $c \in \h_E^r(\M)$ satisfying:
\[
\big\|a\big\|_{\h_E^d } + \big\|b\big\|_{\h_E^c } + \big\|c\big\|_{\h_E^r } \leq \beta_E\big\|x\big\|_{E(\M)}.
\]
\end{lemma}
\begin{proof}
We use   semi-embedding techniques. Using Lemma~\ref{approximation}, we construct  sequence  of operators $(x^{(m)})_{m\geq 1}\subset L_p(\M) \cap L_q(\M)$, sequences  $(a^{(m)} )_{m\geq 1}\subset \h_E^d(\M)$, $(b^{(m)})_{m\geq 1} \subset \h_E^c(\M)$, and $(c^{(m)})_{m\geq 1}\subset \h_E^r(\M)$ such that:   
\begin{itemize}
\item[(i)] $\lim_{m\to \infty} \|x^{(m)}-x\|_{L_p(\M) \cap L_q(\M)}=0$;
\item[(ii)] $x^{(m)} =a^{(m)} + b^{(m)} +c^{(m)}$  for all $m\geq 1$:
\item[(iii)] $\big\|a^{(m)}\big\|_{\h_E^d } + \big\|b^{(m)}\big\|_{\h_E^c } + \big\|c^{(m)}\big\|_{\h_E^r } \leq \eta_E\big\|x\big\|_{E(\M)}$.
\end{itemize}
Let $\rho=\eta_E \|x\|_{E(\M)}$.  Then
 the sequence   $\{\mathcal{D}_d(a^{(m)})\}_{m\geq 1}$ belongs to the  $\rho$-ball of $E(\M \overline{\ot} \ell_\infty)$.  Similarly, $\{U\mathcal{D}_c(b^{(m)})\}_{m\geq 1}$ and $\{U\mathcal{D}_c((c^{(m)})^*)\}_{m\geq 1}$  belong to the $\rho$-ball of $E(\M\overline{\otimes}B(\ell_2(\mathbb{N}^2)))$.

Since    the spaces $L_p(\M \overline{\otimes} \ell_\infty) +L_q(\M \overline{\otimes} \ell_\infty)$  and $L_p(\M\overline{\otimes}B(\ell_2(\mathbb{N}^2)))+ L_q(\M\overline{\otimes}B(\ell_2(\mathbb{N}^2)))$ are  reflexive, we  may assume (after taking subsequences if necessary) that $\{\mathcal{D}_d(a^{(m)})\}_{m\geq 1}$ converges  to $\widetilde{a}$ for weak topology  in $L_p(\M \overline{\otimes} \ell_\infty) +L_q(\M \overline{\otimes} \ell_\infty)$ and both $\{U\mathcal{D}_c(b^{(m)})\}_{m\geq 1}$ and $\{U\mathcal{D}_c((c^{(m)})^*)\}_{m\geq 1}$ converge  (for the weak topology of  $L_p(\M\overline{\otimes}B(\ell_2(\mathbb{N}^2)))+ L_q(\M\overline{\otimes}B(\ell_2(\mathbb{N}^2)))$) to $\widetilde{b}$ and $\widetilde{c}^*$, respectively. 

 As a consequence of the fact that  the inclusion mappings  are semi-embedings, it is clear  that these limits satisfy:
    \[
  \max\Big\{\|\widetilde{a}\|_{E(\M \overline{\otimes} \ell_\infty)}, \|\widetilde{b}\|_{E(\M\overline{\otimes}B(\ell_2(\mathbb{N}^2)))}, \|\widetilde{c}^*\|_{E(\M\overline{\otimes}B(\ell_2(\mathbb{N}^2)))} \Big\} \leq \rho.
  \]
 Using  similar arguments as  in the proof of the previous lemma,  we  may deduce that there exist  $a \in \h_E^d(\M)$, $b \in \h_E^c(\M)$, and $c\in \h_E^r(\M)$ such that  $\widetilde{a}=\mathcal{D}_d(a)$, $\widetilde{b}=U\mathcal{D}_c(b)$, $\widetilde{c}^*=U\mathcal{D}_c(c^*)$, and the previous inequality translates into
 \[
\max\Big\{\|a\|_{\h_E^d}, \|b\|_{\h_E^c},  \|c\|_{\h_E^r}  \Big\} \leq \rho.
\]
Moreover, it is clear  from $\rm{(i)}$ and $\rm{(ii)}$ that $x= a +b +c$. Thus, we have verified Lemma~\ref{approximation-2}.
\end{proof}

To conclude the proof of Part~II,  it is enough to note that  since  $L_p(\M) \cap L_q(\M)$ is dense in $E(\M)$. The assertion that $\big\|x\big\|_{\h_E} \leq \beta_E \big\|x\big\|_{E(\M)}$ for all $x \in E(\M)$  then follows immediately from Lemma~\ref{approximation-2}.

\subsection{The case ${ 2<p_E \leq q_E<\infty}$}
Assume now that $E$ is a  symmetric Banach function space on $(0,\infty)$  satisfying  the Fatou property and $2<p_E \leq q_E <\infty$.

\medskip

\noindent{\bf Part~III.}
We will verify that for every $x\in \h_E(\M)$, 
 \[
\big\|x\|_{E(\M)} \lesssim \big\| x \big\|_{\h_E}.\]
This will be deduced from   Part~II using   duality. Let $E^{\times}$ be the K\"othe dual of $E$.   The noncommutative symmetric space  $E^{\times}(\M)$ is the K\"othe dual of $E(\M)$ in the sense of \cite{DDP3}. Since $E$ has the Fatou property, it follows that  for every $x \in E(\M)$, we have $\|x\|_{E(\M)}= \|x\|_{E^{\times\times}(\M)}$. In particular, the closed unit ball of $E^{\times}(\M)$ is  a norming set for  $E(\M)$.

 It is enough to verify the inequality for   $x \in L_1(\M) \cap \M$. For $\epsilon>0$, choose $y \in E^{\times}(\M)$, with $\|y\|_{E^{\times}(\M)} =1$, and  such that
\[
(1-\epsilon)\big\| x \big\|_{E(\M)} \leq   \T(xy^*).
\] 
From  \cite[Proposition~2.b.2]{LT},  the Boyd indices of $E^\times$ satisfy $1<p_{E^{\times}} \leq  q_{E^{\times}}<2$. Thus, using  Part~II, it follows that $y\in \h_{E^\times}$. We may choose  a decomposition $y= a + b + c$  satisfying:
\[
\big\|a \big\|_{\h_{E^{\times}}^d}  + \big\|b \big\|_{\h_{E^{\times}}^c} +\big\|c \big\|_{\h_{E^{\times}}^r}\leq \kappa_{E^{\times}} +\epsilon.
\]
From  the discussion after the proof of  Part~I, $a$, $b$, and $c$ can be represented by operators from $E^\times(\M)$, which we will still denote by $a$, $b$, and $c$. By density, there exist $N\geq 1$, $\hat{a}$, $\hat{b}$, and $\hat{c}$ in
 $ L_1(\M_N) \cap \M_N$ so that 
 \[
 \big\| a-\hat{a}\big\|_{h_{E^\times}^d} +\big\| a-\hat{a}\big\|_{E^\times(\M)} 
 <\epsilon/3, \]
 \[  \big\| b-\hat{b}\big\|_{h_{E^\times}^c} +\big\| b-\hat{b}\big\|_{E^\times(\M)} 
 <\epsilon/3, \]
  \[ \big\| c-\hat{c}\big\|_{h_{E^\times}^r} +\big\| c-\hat{c}\big\|_{E^\times(\M)} 
 <\epsilon/3.
 \]
Now,  $\T(xy^*)= \T(xa^*) + \T(xb^*) + \T(xc^*)\leq \epsilon \big\|x\big\|_{E(\M)} + |\T(x\hat{a}^*)| + |\T(x\hat{b}^*)| + |\T(x\hat{c}^*)| :=\epsilon\big\|x\big\|_{E(\M)} +{ I} + { II} + {II}$. We estimate ${ I}$, ${ II}$, and ${III}$ separately. Below, we   denote by $\gamma$ and $\tr$  the usual traces on $\ell_\infty$ and $B(\ell_2(\mathbb{N}^2))$, respectively. For ${ I}$, we have the following estimates:
\begin{align*}
{ I} &=\Big| \sum_{n=1}^N \T(dx_n d\hat{a}_n^*)\Big|\\
& =\Big|\T \otimes \gamma \big( (dx_n)_{1\leq n \leq N} . (d\hat{a}_n^*)_{1\leq n\leq N} \big)\Big|\\
&\leq \big\| (dx_n)_{1\leq n\leq N} \big\|_{E(\M \overline{\otimes} \ell_\infty)} . \big\| (d\hat{a}_n)_{1\leq n\leq N} \big\|_{E^{\times}(\M \overline{\otimes} \ell_\infty)}\\
&= \big\|x \big\|_{\h_E^d} . \big\|\hat{a} \big\|_{\h_{E^{\times}}^d}\\
&\leq \big(\epsilon/3 +\big\|a\big\|_{\h_{E^{\times}}^d} \big) \big\|x \big\|_{\h_E^d}.
\end{align*}
 For $ II$, we use the identifications  of $\h_E^c(\M)$  and  $\h_{E^{\times}}^c(\M)$ as  subspaces of $E(\M \overline{\otimes}B(\ell_2(\mathbb{N}^2)))$ and  $E^{\times}(\M \overline{\otimes}B(\ell_2(\mathbb{N}^2)))$, respectively. First, we write ${ II}=\big|\sum_{n=1}^N \T(dx_n d\hat{b}_n^*)\big|$. Since the conditional expectations $\E_k$'s are trace invariant,  we have:
\begin{align*}
 { II} &=\Big|\sum_{n= 1}^N \T(\E_{n-1}(d\hat{b}_n^* dx_n))\Big|\\
&=\Big|\T\Big[\sum_{n= 1}^N \E_{n-1}(d\hat{b}_n^* dx_n)\Big]\Big|.
\end{align*}
We note that $(d\hat{b}_n)_{1\leq n\leq N}$ and $( dx_n)_{1\leq n \leq N}$ are sequences  from $\mathcal{F}$ and therefore  \eqref{key-identity} applies. We may then write
\begin{align*}
II &=\Big|\T \otimes \tr \Big[ U( (d\hat{b}_n)_{1\leq n\leq N})^* U((dx_n)_{1\leq n\leq N}) \Big] \Big|\\
&\leq \Big\| U( (d\hat{b}_n)_{n\geq 1})\Big\|_{E^{\times}(\M \overline{\otimes}B(\ell_2(\mathbb{N}^2)))}
\Big\| U((dx_n)_{n\geq 1}) \Big\|_{E(\M \overline{\otimes}B(\ell_2(\mathbb{N}^2)))}\\
&=\big\| \hat{b} \big\|_{\h_{E^{\times}}^c} . \big\|x \big\|_{\h_{E}^c} \\
&\leq (\epsilon/3 + \big\| b \big\|_{\h_{E^{\times}}^c} ) \big\|x \big\|_{\h_{E}^c}. 
\end{align*}
The proof that $ III \leq  \big(\epsilon/3 +\big\|c \big\|_{\h_{E^{\times}}^r} \big) \big\|x \big\|_{\h_{E}^r}$ is identical so we omit the details.  Combining the above estimates on ${ I}$, ${ II}$, and ${ III}$, we  derive  that 
\[
(1-2\epsilon)\big\|x \big\|_{E(\M)} \leq  (\kappa_{E^{\times}} +2\epsilon) \big\|x \big\|_{\h_E}.
\]
Taking infimum over $\epsilon$  gives the desired inequality.

\bigskip

\noindent{\bf Part~IV.} The remaining case is  to show the reverse inequality $\big\|x\big\|_{\h_E} \lesssim \big\|x\big\|_{E(\M)}$. 
This is an easy  application  of the noncommutative Burkholder inequalities  for  $L_s$-bounded martingales  when $2<s<\infty$ together with  Proposition~\ref{comp-Int}. Details are left to the reader.
\qed

\begin{remark} Our duality argument in Part~III  strongly relies on  the theory of K\"othe dualities for noncommutative symmetric space from \cite{DDP3}. At the time of this writing, we do not know how to incorporate this theory into  martingales Hardy spaces.
\end{remark}

\bigskip

For the case where $\M$ is a finite von Neumann algebra equipped with a normal tracial state $\T$, it is more natural to consider symmetric Banach function spaces defined on the interval $[0,1]$. However, the definition of the diagonal Hardy space uses the infinite von Neumann algebra $\M \overline{\otimes} \ell_\infty$. In this case, we may  consider an  extension  of  symmetric Banach function space $E$ on $[0,1]$  into a symmetric Banach function space on $(0,\infty)$  introduced in \cite{JO-MA-SC-TZ} (see also  \cite{Asta-Sukochev-Wong, LT} for more details).

Let $Z_{E}^2$ be the symmetric space on $(0,\infty)$ of all measurable functions $f$ for which $\mu(f)\ch_{(0,1]} \in E$ and $\mu(f)\ch_{(1,\infty)} \in L_2(0,\infty)$, endowed with the norm 
\[
\big\|f\big\|_{Z_{E}^2} =\max\Big\{ \big\|\mu(f)\ch_{(0,1]}\big\|_E, \Big(\sum_{n=0}^\infty \Big( \int_n^{n+1} \mu_u(f)\ du \Big)^2 \Big)^{1/2} \Big\}.
\]
It was shown in \cite[Theorem~2.f.1]{LT} that if $E$ has nontrivial Boyd indices then $Z_{E}^2$ is isomorphic to $E$. Using the symmetric  Banach function space $Z_{E}^2$ on the diagonal Hardy space, we may state the following variant of Theorem~\ref{main}: 

\begin{theorem}\label{main-finite} Assume that $(\M,\T)$ is a finite von Neumann algebra with $\T$  being a normal tracial  state and $E$ is a symmetric Banach function space on $[0,1]$ satisfying  the Fatou property.  Let $x=(x_n)_{n\geq 1}$ be a  bounded $E(\M)$-martingale. \begin{enumerate}
\item If $1<p_E \leq q_E <2$, then
\begin{equation*}
\big\|x\big\|_{E(\M)} \simeq_E 
\inf\Big\{ \big\| w \big\|_{\h_{Z_{E}^2}^d}  + \big\|  y\big\|_{\h_E^c} + \big\| z\big\|_{\h_E^r}  \Big\}
\end{equation*}
where the infimum runs over all decompositions $x=w +y +z$ with $w$, $y$, and $z$ are martingales.

\item If $2<p_E \leq  q_E <\infty$, then 
\begin{equation*}
\big\|x\big\|_{E(\M)} \simeq_E 
\max\Big\{ \big\| x \big\|_{\h_{Z_{E}^2}^d}  , \big\|  x\big\|_{\h_E^c} , \big\| x\big\|_{\h_E^r}  \Big\}.
\end{equation*}
\end{enumerate}
\end{theorem}

The assumptions of Theorem~\ref{main}  and Theorem~\ref{main-finite} are equivalent to $E\in {\rm Int}(L_p, L_q)$  with $1<p <q<2$ or $2<p<q<\infty$. Indeed, if $E\in {\rm Int}(L_p, L_q)$ then $p\leq p_E \leq q_E \leq q$. We do not know if our results extend to  the case where $E\in {\rm Int}(L_p, L_2)$ for $1<p<2$ or $E\in {\rm Int}(L_2, L_q)$ for $q>2$. On the other hand, since $\h_1(\M) \subset L_1(\M)$, the argument used in Part~I of the proof of Theorem~\ref{main} can be readily  adjusted to provide  the following:

 If $E\in {\rm Int}(L_1, L_q)$ where $q\leq 2$, then there exists a constant  $c_E$ such that
for every martingale $x \in \h_E(\M)$, 
\[
\big\| x\big\|_{E(\M)} \leq c_E \big\|x \big\|_{\h_E}.
\]
Similarly,   
if $E\in {\rm Int}(L_2, L_q)$ where $2<q<\infty$, then there exists a constant  $c_E$ such that
for every  $y \in E(\M)$, 
\[
\big\| y\big\|_{\h_E} \leq c_E \big\|y \big\|_{E(\M)}.
\]

\medskip

We conclude this section by pointing out  that as with the case of  noncommutative Burkholder-Gundy inequalities, no  equivalence  of norms  is known for the case where $1<p_E <2$ and $2<q_E <\infty$.  

\section{Further remarks}


We begin this  section with a short discussion about    the comparison between martingales Hardy spaces and conditioned martingale Hardy spaces associated with general symmetric spaces. Combining  Theorem~\ref{main} with  the main result of \cite{Jiao2}, we may state:

\begin{corollary}\label{davis}
 Let $E$ be a symmetric Banach function  space on $(0,\infty)$ satisfying the Fatou property. Assume that  either $1<p_E \leq q_E < 2$ or $2<p_E \leq q_E<\infty$. Then 
\[
\H_E(\M)= \h_E(\M)=E(\M), 
\]
with equivalent norms.
\end{corollary}

We recall  the noncommutative Davis' decomposition established in \cite{Junge-Mei,Perrin}  which states that 
$\H_1(\M)=\h_1(\M)$. In view of  this equivalence, it seems reasonable to expect  that  the assumption $p_E>1$ is not needed in  the first equality in  Corollary~\ref{davis}. Unfortunately, our interpolation techniques are not efficient enough to apply to the case $p_E=1$. We leave this as an open problem.
\begin{problem}  
Assume that $p_E=1$ and $q_E \leq 2$. Do we have $\H_E(\M)=\h_E(\M)$?
\end{problem}


Our  next result  shows a connection between the diagonal Hardy space and the row/column Hardy spaces.

\begin{proposition}\label{comparison} If  $q_E <2$, then 
$
\max\big\{ \big\|a\big\|_{\H_E^c}, \big\|a\big\|_{\H_E^r} \big\}\leq  c_E\big\|a \big\|_{\h_E^d}$.
\end{proposition}
\begin{proof} Let $1\leq s<2$.
It  is an immediate consequence of the space $L_{s/2}(\M)$ being a $s/2$-normed space   that if $a \in  \h_s^d(\M)$ then $\max\big\{ \big\|a\big\|_{\H_s^c}, \big\|a\big\|_{\H_s^r} \big\}\leq \big\|a \big\|_{\h_s^d} $.   
 The general case can be achieved by interpolations. We  appeal to a result from \cite[Theorem~2 and Remark~4]{Asta-Malig2}  which asserts that if $q_E<q<2$, then $E \in {\rm Int}(L_1, L_q)$. 
 
  When $1\leq s<2$,  it follows from above that the formal identity
 $\iota: \h_s^d(\M) \to \H_s^c(\M)$ is a contraction. If  we denote by $\jmath$ the natural isometry of $\H_s^c(\M)$ into $L_s(\M \overline{\otimes} B(\ell_2(\mathbb{N})))$ then 
 $\jmath  \iota: \h_s^d(\M) \to L_s(\M \overline{\otimes} B(\ell_2(\mathbb{N})))$ is a contraction. We can now deduce from Proposition~\ref{comp-Int} that $\jmath \iota: \h_E^d(\M) \to E(\M \overline{\otimes} B(\ell_2(\mathbb{N})))$ is bounded whose range sits in  $\H_E^c(\M)$. This shows that $\big\|a \big\|_{\H_E^c} \lesssim \big\|a\big\|_{\h_E^d}$. We should point out  here that in general, $\H_1^c(\M)$ is not necessarily complemented in $L_1(\M;\ell_2^c)$ and we  do not know if $\H_E^c(\M) \in {\rm Int}(\H_1^c(\M), \H_q^c(\M))$.  This explains the role of $\jmath$ in our argument. 
\end{proof}
Since for $1\leq s<2$, we have  $\big\|a\big\|_{\H_s^c} \leq 2^{1/s} \big\|a \big\|_{\h_s^c}$  for all $a\in \h_s^c(\M)$ 
(\cite[Theorem 7.1]{JX}), it is reasonable to assume that  the corresponding statements to the row/column are also valid. That is, $\big\|a\big\|_{\H_E^c} \leq c_E \big\|a \big\|_{\h_E^c}$
and  $\big\|a\big\|_{\H_E^r} \leq  c_E \big\|a \big\|_{\h_E^r}$ for some constant $c_E$.
But we were unable to verify  these inequalities at this time.

\medskip

 The following consequence of  Theorem~\ref{simultaneous}  now follows from Proposition~\ref{comparison} and the facts that $\h_p^c(\M) \subset \H_p^c(\M)$ and  $\h_p^r(\M) \subset \H_p^r(\M)$ for $1\leq p<2$. It provides simultaneous decompositions for   martingale Hardy spaces norms that are related to the noncommutative Burkholder-Gundy inequalities from \cite{PX}.

\begin{corollary}\label{simultaneous2}
There exists a family  of constants $\{\kappa_p' : 1<p<2\} \subset \mathbb{R}_+$  satisfying the following:
 if 
 $x\in L_1(\M) \cap L_2(\M)$, then there exist martingales $\displaystyle{y \in \cap_{1<p<2} \H_p^c(\M)}$  and $\displaystyle{z \in \cap_{1<p<2} \H_p^r(\M)}$ such that:
\begin{enumerate}[{\rm(i)}]
\item  $x =y+z $;
\item for every $1<p< 2$, the following inequality holds:
\[
 \big\| y \big\|_{\H_p^c} + \big\| z
\big\|_{\H_p^r}\leq \kappa_p' \big\| x \big\|_p.
\]
\end{enumerate}
\end{corollary}
A direct alternative approach to Corollary~\ref{simultaneous2} is to use the weak-type inequality for square functions  from \cite{Ran18} (see also \cite{PR}) and then follow 
the same line of reasoning as in the proof of Theorem~\ref{simultaneous}. The above result   is  closely related to another type of simultaneous decompositions considered by Junge and Perrin in \cite[Theorem~3.3]{Junge-Perrin}.

\medskip

We note that  Part~{\rm I}  of the proof of Theorem~\ref{main} can  also be deduced from Proposition~\ref{comparison} and the noncommutative Burkholder-Gundy  inequalities for symmetric spaces from   \cite{Dirk-Pag-Pot-Suk} and \cite{Jiao2}.

\medskip

As a final remark, we provide an example  involving Orlicz functions.
Let $\Phi$ be an Orlicz function on $[0,\infty)$ i.e, a continuous increasing   convex function on $[0,\infty)$ with $\Phi(0)=0$ and $\lim_{t \to \infty}\Phi(t)=\infty$.  Two standard indices associated to the Orlicz function $\Phi$ are defined  as follows: let 
\[
M_\Phi(t) =\sup_{s>0} \frac{\Phi(ts)}{\Phi(s)}, \quad t \in [0,\infty)
\]
and  
\[
p_\Phi :=\lim_{t \to 0^+} \frac{\log M_\Phi(t)}{\log t}, \quad q_\Phi:=\lim_{t \to \infty} \frac{\log M_\Phi(t)}{\log t}.
\]
Then $1\leq p_\Phi \leq q_\Phi \leq \infty$.
The Orlicz space $L_\Phi$  is the set of all Lebesgue measurable functions $f$ defined on $(0,\infty)$ such that for some  constant $c>0$,
\[
\int_0^\infty \Phi\Big({|f(t)|}/{c}\Big)\ dt <\infty.
\]
If we equip $L_\Phi$ with the Luxemburg norm:
\[
\big\|f\big\|_{L_\Phi}=\inf\Big\{ c>0: \int_0^\infty \Phi\Big({|f(t)|}/{c}\Big)\ dt \leq 1\Big\},
\]
then $L_\Phi$ is a symmetric Banach function space  with the Fatou property.
Moreover, the Boyd indices of $L_\Phi$ coincide with the indices $p_\Phi$ and $q_\Phi$ (see for instance \cite{Maligranda}). Thus, Theorem~\ref{main} and Corollary~\ref{davis} apply to  martingales in  the noncommutative space $L_\Phi(\M)$ whenever $1<p_\Phi \leq  q_\Phi <2$ or $2<p_\Phi\leq q_\Phi<\infty$. This example also motivates the consideration of the so-called $\Phi$-moment inequalities involving conditioned square functions. This  direction will be explored in a forthcoming article.  We refer to  \cite{Bekjan-Chen,Dirksen, Dirksen-Ricard} for  recent progress on moment inequalities for noncommutative martingales.

\medskip

\noindent {\bf Acknowledgements.}
This work was carried out during  the second-named author's visit to  Miami University. She would like to express her gratitude to the Department of Mathematics of Miami University for its warm hospitality. The authors are indebted to  the referee  for helpful  suggestions that improved the presentation of the paper.


\end{document}